\documentclass[12pt, reqno]{amsart}
\usepackage{amssymb,a4}
\usepackage{amsmath,amsthm,amscd}
\usepackage{amsthm}
\usepackage{amsfonts,amssymb}
\usepackage{mathrsfs}
\usepackage{amsmath}

\newtheorem{theo}{{Theorem}}[section]
\newtheorem{lemm}[theo]{Lemma}
\newtheorem{rema}[theo]{Remark}
\newtheorem{defi}[theo]{Definition}
\newtheorem{coro}[theo]{Corollary}

\numberwithin{equation}{section}

\begin{document}

\title{Whittaker modules for a subalgebra of $\emph{N}$=2 superconformal algebra }

\author{ Naihuan Jing${}^{1,2}$, Pengfa Xu${}^1$ and Honglian Zhang${}^1$}
\maketitle

\begin{center}
\footnotesize
\begin{itemize}
\item[1] Department of Mathematics, Shanghai University, Shanghai 200444, China
\item[2] Department of Mathematics, North Carolina State University,   Ra\-leigh, NC 27695-8205, USA
\end{itemize}
\end{center}
\begin{abstract}
 In this paper, Whittaker modules are studied for a subalgebra $\mathfrak{q}_{\epsilon}$ of the $\emph{N}$=2 superconformal algebra.
 The Whittaker modules are classified by central characters.
 Additionally, criteria for the irreducibility of the Whittaker modules are given.\\

\end{abstract}

%


\section{Introduction}

Whittaker modules, named after relations with the Whittaker equation in number theory, arise naturally in the representation theory of Lie algebras.
They were first introduced in the context of constructing representations for $sl_2(\mathbb{C})$ by Arnal and Pinzcon in \cite{5}.
Later, Kostant extended the concept to all finite-dimensional complex semisimple Lie algebras (denoted as $\mathfrak{g}$), linking Whittaker modules bijectively to the ideals of the center $Z(\mathfrak{g})$ of the Lie algebras $\frak g$ in \cite{1}. This underscored the significance of Whittaker modules, including their central role in the classification of irreducible modules for $sl_2(\mathbb{C})$.
Recently, there has been an intensive investigation of Whittaker modules for many infinite-dimensional Lie algebras such as the Virasoro algebra \cite{17,13,11}, generalized Weyl algebras \cite{36}, Heisenberg algebras \cite{14}, Heisenberg-Virasoro algebra \cite{29}, Schrodinger-Virasoro algebra \cite{27}, Schrodinger-Witt algebra \cite{7}, \emph{W}-algebra $W(2,2)$ \cite{25}, and affine Lie algebra $A_1^{(1)}$ \cite{24}. Some of these studies have also provided for Lie algebras with triangular decompositions (see~\cite{15,12,9,8} and references therein).
Moreover, substantial research has been conducted on the quantum deformation of Whittaker modules, as well as on the modules induced from Whittaker modules
(see \cite{33,6} for more details).

A general categorical framework for Whittaker modules was proposed in \cite{21,19}. Notably, degenerate Whittaker vectors of the Virasoro algebra have emerged naturally in the AGT conjecture in physics \cite{20}. In \cite{18}, an explicit formula for degenerate Whittaker vectors has been derived in terms of the Jack symmetric functions.

Expanding the scope, Whittaker modules and Whittaker categories have been generalized to finite-dimensional simple Lie superalgebras based on certain representations of nilpotent finite-dimensional Lie superalgebras \cite{22}. Recently extensive study of Whittaker modules have been done
for a variety of
infinite-dimensional Lie superalgebras. In \cite{23}, simple Whittaker modules for the super-Virasoro algebras, including the Neveu-Schwarz algebra and the Ramond algebra, were classified as well as criteria for the irreducibility of these modules. In \cite{26}, C. Chen classified simple Whittaker modules for classical Lie superalgebras in terms of their parabolic decompositions. Whittaker modules have also garnered significant attention in relation to other objects (cf. \cite{ 22, 28, 37, 38}). In addition, investigations into Whittaker modules have been conducted within the framework of vertex operator algebra theory (see \cite{32,24,31,30}).
Overall, the study of Whittaker modules extends across various contexts, spanning both Lie algebras and Lie superalgebras, and incorporating different mathematical frameworks for analysis and classification.

The $N=2$ superconformal algebras can be divided into four sectors: the Ramond sector, the Neveu-Schwarz sector, the topological sector, and the twisted sector. In particular, the Lie superalgebra $\mathfrak{q}_{\epsilon}$ is a subalgebra of both the Ramond and Neveu-Schwarz sectors of the $N=2$ superconformal algebra, with its even part being the twisted Heisenberg-Virasoro algebra. The study of cuspidal modules for $\mathfrak{q}_{\epsilon}$ was carried out in \cite{34}, which plays a pivotal role in the classification of all simple cuspidal weight modules for the $N=2$ superconformal algebra.

In this paper, we investigate Whittaker modules for the Lie superalgebra $\mathfrak{q}_{\epsilon}$. We provide a classification of all finite-dimensional modules over the subalgebra $\mathfrak{q}_{\epsilon}^+$, demonstrating that they are one-dimensional. Moreover, we also
classify simple Whittaker modules based on their central characters and establish a criterion for the irreducibility of Whittaker modules.

Throughout this paper, we adopt the conventions that $\mathbb{C}$ denotes the set of complex numbers, $\mathbb{C}^*$ represents the set of nonzero complex numbers, $\mathbb{N}$ corresponds to the set of non-negative integers, $\mathbb{Z}_+$ signifies the set of positive integers, and $\mathbb{Z}$ denotes the set of integers. For convenience, unless explicitly stated, all elements within superalgebras and modules are assumed to be homogeneous.

\section{Preliminary}

In this section, we list some notations and known facts that will be used throughout the paper. A superspace is a vector space with a $\mathbb{Z}_2$-grading, denoted as $\emph{V}=\emph{V}_{\bar{0}} \oplus \emph{V}_{\bar{1}}$. Typically, the elements in $\emph{V}_{\bar{0}}$ are referred to as even, while the elements in $\emph{V}_{\bar{1}}$ are called odd.
A Lie superalgebra is a superspace $\mathfrak{L}=\mathfrak{L}_{\bar{0}} \oplus \mathfrak{L}_{\bar{1}}$ equipped with a commutation operation $[\cdot, \cdot]$ that satisfies the following conditions,
$$
[x_{1},x_{2}]=-(-1)^{p_{1}p_{2}}[x_{2},x_{1}],
$$
$$
(-1)^{p_{1}p_{2}}[[x_{1},x_{2}],x_{3}]+(-1)^{p_{2}p_{3}}[[x_{2},x_{3}],x_{}]+(-1)^{p_{3}p_{1}}[[x_{3},x_{1}],x_{2}]=0,
$$
here, $x_{1} \in \mathfrak{L}_{\bar{p}_{1}}$, $x_{2} \in \mathfrak{L}_{\bar{p}_{2}}$, $x_{3} \in \mathfrak{L}_{\bar{p}_{3}}$.


\begin{defi}\label{D 2.1}
Let $\mathfrak{L}$ be a Lie superalgebra. An $\mathfrak{L}$-module is a $\mathbb{Z}_{2}$-graded vector space $V=V_{\bar{0}} \oplus V_{\bar{1}}$ together with a bilinear map, $(x,v) \mapsto xv$, denoted as $\mathfrak{L} \times V \rightarrow V$, such that it satisfies the following condition,

$$
x(yv)-(-1)^{p_{1}p_{2}}y(xv)=[x,y]v,
$$

for all $x \in \mathfrak{L}_{\bar{p}{1}}$, $y \in \mathfrak{L}_{\bar{p}{2}}$, $v \in V$, and $\mathfrak{L}_{\sigma} V_{\tau} \subseteq V_{\sigma\tau}$ for all $\sigma,\tau \in \mathbb{Z}_{2}$.
\end{defi}

\begin{defi}\label{D 4.1}
Let $\mathfrak{L}$ be a Lie superalgebra, $V$ a $\mathfrak{L}$-module and $x \in \mathfrak{L}$. If for any $v \in V$ there exists $n \in \mathbb{Z}_+$ such that $x^n v=0$, we say that $x$ acts on $V$ locally nilpotent. The action of $\mathfrak{L}$ on $V$ is locally nilpotent if for any $v\in V$ there exists $n \in \mathbb{Z}_+$ such that $\mathfrak{L}^n v=0$.
\end{defi}
The $\emph{N}$=2 superconformal algebra $\mathfrak{g}_{\epsilon}$ is a Lie superalgebra over $\mathbb{C}$ with the basis $\{ L_{m}, H_{m}, G_{r}^{\pm}, C \mid m, n\in \mathbb{Z}~{\rm and}~r \in \mathbb{Z}+\epsilon, \epsilon=0~{\rm or}~ \frac{1}{2}~\}$ with the central element $C$ and
subject to the following commutation relations:
\begin{align}\label{e:re1}
&{[L_{m},L_{n}]=(n-m)L_{m+n}+\frac{1}{12}\delta_{m+n,0}(n^{3}-n)C,}\\ \label{e:re2}
&{[H_{m},H_{n}]=\frac{1}{3}m \delta_{m+n,0}C,}\\ \label{e:re3}
&{[L_{m},H_n]=nH_{m+n},}\\ \label{e:re4}
&{[L_m,G_{p}^{\pm}]=(p-\frac{m}{2})G_{p+m}^{\pm},}\\ \label{e:re5}
&{[H_m,G_{p}^{\pm}]=\pm G_{p+m}^{\pm},}\\ \label{e:re6}
&{[G_{p}^{+},G_{q}^{-}]=-2L_{p+q}+(p-q)H_{p+q}+\frac{1}{3}(p^2-\frac{1}{4})\delta_{p+q,0}C,}\\ \label{e:re7}
&{[G_{p}^{\pm},G_{q}^{\pm}]=0.}
\end{align}
Here $\mathfrak{g}_{0}$ is called the $N$=2 Ramond algebra and $\mathfrak{g}_{\frac{1}{2}}$ is called the $N$=2 Neveu-Schwarz algebra.
\begin{defi}\label{D 2.2}
Let $\mathfrak{q}_{\epsilon}$ is the subalgebra of $\mathfrak{g}_{\epsilon}$ with a basis $\{ L_{m}, H_{m}, G_{r}, C \mid m, n\in \mathbb{Z}~{\rm and}~r \in \mathbb{Z}+\epsilon, \epsilon=0~{\rm or}~ \frac{1}{2}~\}$ satisfying \eqref{e:re1}-\eqref{e:re5} and \eqref{e:re7}, where only relations for $G_r=G_r^+$ are involved.
\end{defi}
Note that $\mathfrak{q}_{\epsilon}$ is $\mathbb{Z}_{2}$-graded
with 
$\mathfrak{q}_{\epsilon}^{\bar{0}}$ := $\mathbb{C} \{ \emph{L}_{m}, \emph{H}_{m}, \emph{C} \mid \emph{m} \in \mathbb{Z} \}$ and $\mathfrak{q}_{\epsilon}^{\bar{1}}$ := $\mathbb{C} \{ \emph{G}_r\mid \emph{r} \in \mathbb{Z}+\epsilon \}$.

By definition, the even part of $\mathfrak{g}_{\epsilon}$ is
the twisted Heisenberg-Virasoro algebra, where the Virasoro algebra (Vir) can be viewed as a subalgebra
with basis $\{ \emph{L}_{m}, \emph{C} \mid m \in \mathbb{Z} \}$ with the central element $C$
subject to the relation \eqref{e:re1}.

The Lie superalgebra $\mathfrak{q}_{\epsilon}$ also has a $(1-\epsilon)\mathbb{Z}$-grading by the eigenvalues of the adjoint action of $L_0$. It follows that $\mathfrak{q}_{\epsilon}$ has the following triangular decomposition,
$$
\mathfrak{q}_{\epsilon}=\mathfrak{q}_{\epsilon}^{+} \oplus \mathfrak{q}_{\epsilon}^{-} \oplus \mathfrak{q}_{\epsilon}^{0},
$$
where
$$
\begin{aligned}
&{\mathfrak{q}_{\epsilon}^{+}={\rm span}_{\mathbb{C}} \{L_m, H_n, G_r \mid m, n \in \mathbb{Z}_{+}~{\rm and}~r \in \mathbb{Z}_{+}-\epsilon \}},\\
&{\mathfrak{q}_{\epsilon}^{-}={\rm span}_{\mathbb{C}} \{L_{-m}, H_{-n}, G_{-r} \mid m, n \in \mathbb{Z}_{+}~{\rm and}~r \in \mathbb{Z}_{+}-\epsilon \}},\\
&{\mathfrak{q}_{\epsilon}^{0}={\rm span}_{\mathbb{C}} \{C,L_0,H_0, G_0 \}}.
\end{aligned}
$$

Set
$$
\begin{aligned}
\mathfrak{b}_{\epsilon}^-&{=\mathfrak{q}_{\epsilon}^{-}\oplus\mathfrak{q}_{\epsilon}^{0}},\\
\mathfrak{b}_{\epsilon}^+&{=\mathfrak{q}_{\epsilon}^{+}\oplus\mathfrak{q}_{\epsilon}^{0}}.
\end{aligned}
$$

Now, let us introduce the concepts of  Whitaker modules and Whitaker vectors.

\begin{defi}\label{D 2.3}
For a $\mathfrak{q}_{\epsilon}$-module V, a vector $v \in V$ is a Whittaker vector if $xv = \psi(x) v$ for 
a complex function $\psi$ on $\mathfrak{q}_{\epsilon}^{+}$. Furthermore, if $v$ generates V, we call V a Whittaker module of type $\psi$ and $v$ a cyclic Whittaker vector of V.
\end{defi}
\begin{lemm}\label{L 2.4}
Let $\psi:\mathfrak{q}_{\epsilon}^+ \to \mathbb{C}$ be a Lie superalgebra homomorphism, then
$$
\psi(L_m)=\psi(H_n)=\psi(G_r)=0
$$
for $m,n \in \mathbb{Z}_+,r \in \mathbb{Z}_{+}-\epsilon, m\geq3 ,n\geq2$ and $r\geq 1-\epsilon$.
\end{lemm}

\begin{proof}
By the commutator relations in Definition \ref{D 2.2}, we can derive
$$
\psi(L_3)=\psi([L_2, L_1])=[\psi(L_2), \psi(L_1)]=0,
$$
$$
\psi(H_2)=\psi([L_1, H_1])=[\psi(L_1), \psi(H_1)]=0,
$$
$$
2\psi(G_{1-\epsilon})^2=[\psi(G_{1-\epsilon}), \psi(G_{1-\epsilon})]=\psi([G_{1-\epsilon}, G_{1-\epsilon}])=0.
$$
Hence $\psi(G_{1-\epsilon})$=0. By induction we have $\psi(L_m)=\psi(I_n)=\psi(G_r)=0$ for $m,n\in \mathbb{Z}_{+},r \in \mathbb{Z}_{+}-\epsilon, m\geq3 ,n\geq2$ and $r\geq 1-\epsilon$.
\end{proof}

For a classical finite-dimensional complex semisimple Lie algebra $\mathfrak{L}$, a Whittaker module is defined based on an algebra homomorphism from the positive nilpotent subalgebra $\mathfrak{L}^+$ to $\mathbb{C}$ (as described in \cite{1}). This homomorphism is required to be nonsingular, meaning that it takes nonzero values on the Chevalley generators of $\mathfrak{L}^+$.

In our current setting, the elements $L_1, L_2, H_1, G_1 \in \mathfrak{q}_{\epsilon}^{+}$ generate $\mathfrak{q}_{\epsilon}^{+}$. Therefore, we assume the existence of a Lie superalgebra homomorphism $\psi: \mathfrak{q}_{\epsilon}^{+} \rightarrow \mathbb{C}$ with the property that $\psi(L_1), \psi(L_2), \psi(H_1) \neq 0$.
Let $S(C)$ denote the symmetric algebra generated by $C$, which consists of polynomials in $C$. It is evident that $S(C)$ is  a subset of $Z(U(\mathfrak{q_{\epsilon}}))$, the center of the universal enveloping algebra $U(\mathfrak{q_{\epsilon}})$.

We define a pseudopartition $\lambda$ to be a non-decreasing sequence of non-negative integers,
$$
\lambda=( 0 \leq \lambda_1 \leq \lambda_2 \leq \cdots \leq
\lambda_m).
$$
The set of pseudopartitions is denoted by $\mathcal{P}$. Similarly, we define a strict pseudopartition $\lambda$ as a strictly increasing sequence of non-negative integers,
$$
\lambda=( 0 \leq \lambda_1 < \lambda_2 < \cdots < \lambda_m).
$$
We denote the set of strict pseudopartitions as $\mathcal{Q}$, which is a subset of $\mathcal{P}$.

An alternative representation of pseudopartition is to display the multiplicities by writing
$$
\lambda = (0^{\lambda(0)}, 1^{\lambda(1)}, 2^{\lambda(2)}, \ldots),
$$
where $\lambda(k)$ denotes the number of times $k$ appearing in the pseudopartition. When $k$ is sufficiently large, $\lambda(k)=0$. It is worth noting that if $\mu \in \mathcal{Q}$, then $\mu(k)$ takes values in $\{0, 1\}$.

For $\lambda, \omega \in \mathcal{P},\mu\in \mathcal{Q}$, we define the elements $L_{-\lambda},H_{-\omega},G_{-\mu-\epsilon} \in U(\mathfrak{q_{\epsilon}})$ as follows,
$$
\begin{aligned}
L_{-\lambda}&{=L_{-\lambda_m} \cdots L_{-\lambda_2}  L_{-\lambda_1}= \cdots L_{-1}^{\lambda(1)} L_0^{\lambda(0)}},\\
H_{-\omega}&{=H_{-\omega_n} \cdots H_{-\omega_2}  H_{-\omega_1}= \cdots H_{-1}^{\omega(1)} H_0^{\omega(0)}},\\
G_{-\mu-\epsilon}&=G_{-\mu_r-\epsilon}\cdots G_{-\mu_2-\epsilon} G_{-\mu_1-\epsilon} = \cdots G_{-1-\epsilon}^{\mu(1)} G_{-\epsilon}^{\mu(0)}.
\end{aligned}
$$

We define the following notations,
$$
\begin{aligned}
|\lambda |\quad~&{= \lambda_1 +  \lambda_2 + \cdots + \lambda_s=\sum_{i\ge0}i\lambda(i) \quad \mbox{(the size of $\lambda$)}},\\
|\mu+\epsilon|\quad~&{= (\mu_1+\epsilon) +  (\mu_2+\epsilon) + \cdots + (\mu_r+\epsilon)=|\mu|+r\epsilon},\\
l(\lambda)\quad~&{=\lambda(0) + \lambda(1) + \cdots \quad \mbox{(the $\#$ of parts of $\lambda$)}},
\end{aligned}
$$
$$
|\lambda+\omega+\mu+\epsilon|\quad~=|\lambda|+|\omega|+|\mu+\epsilon|.
$$
We define $\overline{0}= (0^0, 1^0, 2^0, \ldots)$, and write $L_{\overline{0}}=H_{\overline{0}}=G_{\overline{0}-\epsilon}=1 \in U(\mathfrak{q_{\epsilon}})$. In the following, we treat $\overline{0}$ as an element of both $\mathcal{P}$ and $\mathcal{Q}$. For any $\lambda,\omega \in \mathcal{P}$, $\mu \in \mathcal{Q}$ and $p(C) \in S(C)$, $$p(C)L_{- \lambda}I_{-\omega}G_{-\mu-\epsilon} \in U(\mathfrak{q_{\epsilon}})_{-|\lambda|-|\omega|-|\mu+\epsilon|},$$ where $U(\mathfrak{q_{\epsilon}})_{-|\lambda|-|\omega|-|\mu+\epsilon|}$ is the
$-|\lambda|-|\omega|-|\mu+\epsilon|$-weight space of $U(\mathfrak{q_{\epsilon}})$ under the adjoint action.

For a given $\psi: \mathfrak{q}_{\epsilon}^{+} \rightarrow \mathbb{C}$, we define $\mathbb{C}_\psi$ to be the one-dimensional $\mathfrak{q}_{\epsilon}^{+}$-module with the action $x \alpha = \psi(x) \alpha$ for $x \in \mathfrak{q}_{\epsilon}^{+}$ and $\alpha \in \mathbb{C}$. The universal Whittaker module $M_{\epsilon}(\psi)$ is then defined as
$$
M_{\epsilon}(\psi)= U(\mathfrak{q_{\epsilon}}) \otimes_{U(\mathfrak{q}_{\epsilon}^{+})} \mathbb{C}_{\psi}.
$$

It is evident that $M_{\epsilon}(\psi)$ is a Whittaker module of type $\psi$ with the cyclic Whittaker vector $w = 1 \otimes 1 \in M_{\epsilon}(\psi)$. By the Poincare-Birkhoff-Witt (PBW) theorem,
$U(\mathfrak{b}_{\epsilon}^-)$ has a basis $$\{C^t L_{- \lambda}H_{-\omega}G_{-\mu-\epsilon} \mid \lambda,\omega \in \mathcal{P}, \mu \in \mathcal{Q}, t\in\mathbb{N}\}.$$ Thus, $M_{\epsilon}(\psi)$ has a basis given by
$$
\{C^t L_{- \lambda}H_{-\omega}G_{-\mu-\epsilon}w \mid \lambda,\omega \in \mathcal{P}, \mu \in \mathcal{Q}, t\in\mathbb{N}\}.$$
Moreover, we have $uw\neq 0$ whenever $0\neq u \in U(\mathfrak{b}_{\epsilon}^-)$. We define the degree of $C^t L_{- \lambda}H_{-\omega}G_{-\mu-\epsilon} w$ to be $|\lambda|+|\omega|+|\mu+\epsilon|$. For any $0 \neq v \in M_{\epsilon}(\psi)$, we define ${\rm max~deg}(v)$ to be the maximum degree among the nonzero components of homogeneous degree, and we define  ${\rm max~deg}(0) = -\infty$. Furthermore, $\max_{L_0} (v)$ is defined as the maximum value of $\lambda (0)$ among the terms $C^t L_{- \lambda}H_{-\omega}G_{-\mu-\epsilon} w$ with nonzero coefficient.

For $\xi \in \mathbb{C}$, we define the quotient module
$$L_{\epsilon}(\psi, \xi) = M_{\epsilon}(\psi) / (C - \xi)M_{\epsilon}(\psi),$$
and we denote the canonical homomorphism by  $\pi: M_{\epsilon}(\psi) \rightarrow L_{\epsilon}(\psi, \xi)$. For any $u \in M_{\epsilon}(\psi)$, we use  $\bar{u}$ denote $\pi(u)$. Similar to Lemma in \cite{2} we have the following universal property of $M_{\epsilon}(\psi)$.

\begin{lemm}\label{L 2.5}
Given the fixed $\psi$ and $M_{\epsilon}(\psi)$ as described above, we have the following properties.
\begin{itemize}
\item[(i)] If $V$ is a Whittaker module of type $\psi$ with cyclic Whittaker vector $w_V$,  then there exists a surjective map $\varphi: M_{\epsilon}(\psi) \rightarrow V$ that maps $w=1 \otimes 1$ to $w_V$.
\item[(ii)] Suppose $M$ is a Whittaker module of type $\psi$ with cyclic Whittaker vector $w_M$, and for every Whittaker module $V$ of type $\psi$ with cyclic Whittaker vector $w_V$ there exists a surjective homomorphism $\theta: M \rightarrow V$ satisfying $\theta (w_M) = w_V$. In that case, we have $M \cong M_{\epsilon}(\psi)$.
\end{itemize}
\end{lemm}

\begin{proof}
To prove (i), let $v \in M_{\epsilon}(\psi)$ be given, and consider  $u \in U(\mathfrak{q_{\epsilon}})$ such that $v = uw$.  We define $\varphi : M_{\epsilon}(\psi) \to V$ by $\varphi (v) = uw_V$. Since $w_V$ is a cyclic Whittaker vector of $V$, we only need to show that $\varphi$ is well-defined. To do this, it suffices to show that if $uw = 0$, then $uw_V = 0$.

For $u \in U(\mathfrak{q_{\epsilon}})$, we can write $u = \sum_\alpha b_\alpha n_\alpha$ , where $b_\alpha \in U(\mathfrak{b}_{\epsilon}^-)$ and $n_\alpha \in U(\mathfrak{q}_{\epsilon}^+)$. Therefore, we have $uw = \sum_\alpha b_\alpha n_\alpha w = \sum_\alpha \psi (n_\alpha) b_\alpha w$. By the Poincare-Birkhoff-Witt (PBW) theorem, if $\sum_\alpha \psi (n_\alpha) b_\alpha \in U(\mathfrak{b}_{\epsilon}^-)$ annihilates $w$, then it must be $\sum_\alpha \psi (n_\alpha) b_\alpha = 0$.~Since $w_V \in V$ is a Whittaker vector in $V$, we deduce
$$0 = \left( \sum_\alpha \psi (n_\alpha) b_\alpha \right) w_V = \sum_\alpha b_\alpha \psi (n_\alpha) w_V = \sum_\alpha b_\alpha n_\alpha w_V = uw_V,$$
which proves the well-definedness of $\varphi$.

For part (ii), we choose $V=M_{\epsilon}(\psi)$. By (i), there exists a surjective map $\varphi: M_{\epsilon}(\psi) \rightarrow M$ that taking $w=1 \otimes 1$ to $w_M$. Moreover, since for any Whittaker module $V$ of type $\psi$ with cyclic Whittaker vector $w_V$, there exists a surjective homomorphism $\theta: M \rightarrow M_{\epsilon}(\psi)$ with $\theta (w_M) = w$, we can apply this to $V = M_{\epsilon}(\psi)$ and obtain a surjective homomorphism $\theta: M \rightarrow M_{\epsilon}(\psi)$ satisfying $\theta(w_M) = w$. For any $v = uw \in M_{\epsilon}(\psi)$, using the properties of module homomorphisms, we have
$$
\theta(\varphi(v)) = \theta(uw_M) = u\theta(w_M) = uw = v,
$$
which implies $\theta \circ \varphi = \textbf{1}_{M_\psi}$. Consequently, $M$ is isomorphic to $M_{\epsilon}(\psi)$.
\end{proof}

These properties characterize the universal property of $M_{\epsilon}(\psi)$ and establish its uniqueness among Whittaker modules of type $\psi$ with the specified cyclic Whittaker vector.

\section{Simple Modules Over $\mathfrak{q_{\epsilon}}^+$}

 In this section, we will present a classification of all finite-dimensional simple modules over the subalgebra $\mathfrak{q}_{\epsilon}^{+}$.

\begin{theo}\label{P 4.10}
Let $\psi: \mathfrak{q}_{\epsilon}^+ \rightarrow \mathbb{C}$ be a Lie superalgebra homomorphism and $A(\psi)=\mathbb{C}w_{\psi}$ be a one-dimensional vector space with the action  $xw_{\psi}=\psi(x)w_{\psi}$ for all $x \in \mathfrak{q}_{\epsilon}^+$.  Then  every finite-dimensional simple module over $\mathfrak{q}_{\epsilon}^+$ is isomorphic to such an $A(\psi)$.
\end{theo}

\begin{proof}
Assume that $V$ is a finite-dimensional simple $\mathfrak{q}_{\epsilon}^+$-module. Set Vir$^+={\rm span}_{\mathbb{C}}\{L_m \mid m \in \mathbb{Z}_+ \}$. We can view $V$  as a finite-dimensional ${\rm Vir}^+$-module. Let $W$ be a simple ${\rm Vir}^+$-submodule of $V$. According to \cite[Prop.12]{19}, $W$ must be one-dimensional. Hence, $W=\mathbb{C} w$  and we have $L_nw=c_nw$ for all $n\in\mathbb{Z}_+$, where $c_1, c_2\in\mathbb{C}^{*}$ and $c_i=0$ for all $i\ge3$.

Let $\mathfrak{q}_{\epsilon}^H$ be the  subalgebra of $\mathfrak{q}_{\epsilon}^+$ generated by $H_n$ and $\mathfrak{q}_{\epsilon}^G$ be the subalgebra of $\mathfrak{q}_{\epsilon}^+$ generated by $G_n$ for $n \in \mathbb{Z}_+-\epsilon$. Let $U=U(\mathfrak{q}_H)w$ be a $\mathfrak{q}_{\epsilon}^H$-submodule of $V$,  Since $U$ is finite-dimensional, by Definition \ref{D 2.2}, we have $H_nH_m w=H_mH_n w$ for all $m,n \in \mathbb{Z}_+$. By
Schur's Lemma (see \cite{4}), we  deduce that $H_n$ acts by the scalar $d_n$ on $U$ for all $n \in \mathbb{Z}_+$. Therefore, $U=\mathbb{C}w$. As a consequence, we have $H_n w=d_n w$ for all $n\in\mathbb{Z}_+$, where $d_1 \in\mathbb{C}$ and $d_i=0$  for $i\geq 2$. In fact, for $i\geq2$,
$$
H_i w=[L_{i-1},H_1]w=c_{i-1}d_1w-d_1c_{i-1}w=0.
$$
Similarly, we can show that $U(\mathfrak{q}_{\epsilon}^G)w=\mathbb{C}w$. Consequently, $W$ is a nonzero $\mathfrak{q}_{\epsilon}^+$-submodule of $V$, which implies $$V=W=\mathbb{C}w\cong A(\psi).$$
\end{proof}

Similar to Corollary 3.5 and Corollary 4.11 in \cite{23}, we can establish the following results.
\begin{coro}\label{C 4.11}
Let $V$ be a finite-dimensional simple module over $\mathfrak{q}_{\epsilon}^+$, where the element $C$ acts as a scalar $\xi\in \mathbb{C}$. Define the induced $\mathfrak{q}_{\epsilon}$-module $M_{\epsilon}(V,\xi)=U(\mathfrak{q}_{\epsilon})\otimes_{U(\mathfrak{q}_{\epsilon}^+\oplus\mathbb{C}C)}V$. Then, there exists a Lie superalgebra homomorphism $\psi : \mathfrak{q}_{\epsilon}^+ \to \mathbb{C}$ such that $M_{\epsilon}(V,\xi)\cong L_{\epsilon}(\psi, \xi)$ as modules over $\mathfrak{q}_{\epsilon}$.
\end{coro}

\begin{coro}\label{C 3.3}
Let $V$ be a finite-dimensional module over $\mathfrak{q}_{\epsilon}^+$, where the element $C$ acts as a scalar $\xi\in \mathbb{C}$. Then $M_{\epsilon}(V,\xi)$ is a simple module for $\mathfrak{q}_{\epsilon}$ if and only if $V$ is a simple module for $\mathfrak{q}_{\epsilon}^+$.
\end{coro}

\section{Whittaker vectors in $M_{\epsilon}(\psi)$ and $L_{\epsilon}(\psi, \xi)$}

Consider the fixed Lie superalgebra homomorphism $\psi : \mathfrak{q}_{\epsilon}^+ \to \mathbb{C}$ such that $\psi(L_1), \psi(L_2), \psi(H_1) \neq 0$. Let $w = 1 \otimes 1 \in M_{\epsilon}(\psi)$ be the Whittaker vector as defined above. In this section, we will determine the Whittaker vectors in the modules $M_{\epsilon}(\psi)$ and $L_{\epsilon}(\psi, \xi)$.

\begin{lemm}\label{L 3.1}
Let $w$ be an arbitrary vector of $M_{\epsilon}(\psi)$, and let $w'=u w$ for some $u \in U(\mathfrak{b}_{\epsilon}^-)$. Then, for any $x \in \mathfrak{q}_{\epsilon}^+$,  we have
$$
(x-\psi(x))w'=[x, u]w.
$$
\end{lemm}

\begin{proof}
For $x\in\mathfrak{q}_{\epsilon}^+$, by Lemma \ref{L 2.4} we have
$$
[x,u]w=xuw- uxw=xuw-\psi(x)uw=(x-\psi(x))uw=(x-\psi(x))w'.
$$
\end{proof}

Note that if $w'$ is a Whittaker vector, then $(x-\psi(x))w'=0$ for $x\in\mathfrak{q}_{\epsilon}^+$.

\begin{lemm}\label{L 3.4}
For any $(\lambda, \omega, \mu) \in \mathcal{P}\times \mathcal{P}\times \mathcal{Q}$, $m, k\in \mathbb{N}$, we have the following.
\begin{itemize}
\item[(i)]
\begin{flushleft}
${\rm max~deg}([L_m,L_{-\lambda}H_{-\omega}G_{-\mu-\epsilon}]w)\leq|\lambda+\omega+\mu+\epsilon|-m+2,$
\end{flushleft}
\begin{flushleft}
${\rm max~deg}([L_m,H_{-\omega}G_{-\mu-\epsilon}]w)\leq|\omega+\mu+\epsilon|-m+1;$
\end{flushleft}
\item[(ii)]If $\omega(i)=\lambda(i)=0$ for all $0\leq i\leq k,$ then
$$
{\rm max~deg}([L_{k+1},L_{-\lambda}H_{-\omega}G_{-\mu-\epsilon}]w) \leq|\lambda+\omega+\mu+\epsilon|-k-1;
$$
\item[(iii)]If $\omega(j)=\lambda(i)=0$ for all $0\leq i\leq k,$ $0\leq j<k$ and $\omega(k)\neq0,$ then
$$
[L_{k+1},L_{-\lambda}H_{-\omega}G_{-\mu-\epsilon}]w=v-k\psi(H_1)\omega(k)L_{-\lambda}H_{-\omega{'}}G_{-\mu-\epsilon}w,
$$
where ${\rm max~deg}(v)<|\lambda+\omega+\mu+\epsilon|-k$ and $\omega{'}$ satisfies that
$\omega{'}(i)=\omega(i)$ for all $i\neq k$ and $\omega{'}(k)=\omega(k)-1.$
\end{itemize}
\end{lemm}

\begin{proof}
To prove (i), suppose $\lambda=(\lambda_1,\lambda_2,\cdots,\lambda_s)$, $\omega=(\omega_1,\omega_2,\cdots,\omega_r)$, $\mu=(\mu_1,\mu_2,\cdots,\mu_p)$. Note that $[E_m,-]$ is a derivation of U($\mathfrak{q}_{\epsilon}$) for any $$E_m\in \{L_m, H_m, G_{m-\epsilon}\mid m \in \mathbb{Z}_+\},$$ we have
$$
\begin{aligned}
{[L_m,L_{-\lambda}H_{-\omega}G_{-\mu-\epsilon}]}&{=\sum_{i=1}^{s}L_{-\lambda_s}\cdots[L_m,L_{-\lambda_{i}}]\cdots
L_{-\lambda_1}H_{-\omega}G_{-\mu-\epsilon}}\\
&{+\sum_{j=1}^{r}L_{-\lambda}H_{-\omega_r}\cdots[L_m,H_{-\omega_{j}}]\cdots
H_{-\omega_1}G_{-\mu-\epsilon}}\\
&{+\sum_{k=1}^{p}L_{-\lambda}H_{-\omega}G_{-\mu_p-\epsilon}\cdots[L_m,G_{-\mu_k-\epsilon}]\cdots G_{-\mu_1-\epsilon}}\\
&{=\sum p_{\lambda',\omega',\mu'}(C)L_{-\lambda'}H_{-\omega'}G_{-\mu'-\epsilon}}\\
&{+\sum p_{\lambda'',\omega'',\mu'',E_n}(C)L_{-\lambda''}H_{-\omega''}G_{-\mu''-\epsilon}E_n},
\end{aligned}
$$
where $\lambda',\omega',\mu',\lambda'',\omega'',\mu''$ and $n$ satisfy
$$
-|{\lambda{'}}+\omega'+\mu{'}+\epsilon|= -|{\lambda{''}}+\omega''+\mu{''}+\epsilon|+n
=-|\lambda+\omega+\mu+\epsilon|+m,
$$
and $E_n=L_n,H_n$ or $G_n$. Since $L_iw=H_jw=G_kw=0$ for $i>2, j>1$ and $k>0$, we have
$$
\begin{aligned}
&{{\rm max~deg}([L_m,L_{-\lambda}H_{-\omega}G_{-\mu-\epsilon}]w)}\\
&{={\rm max~deg}(\sum p_{\lambda'',\omega'',\mu'',E_n}(C)L_{-\lambda''}H_{-\omega''}G_{-\mu''-\epsilon}E_nw)}\\
&{\leq |\lambda+\omega+\mu+\epsilon|-m+2}.
\end{aligned}
$$
The second inequality can be proved in a similar way.

To prove (ii), note that if $(\lambda,\omega)=(\bar{0},\bar{0})$, then
$$
[L_{k+1},L_{-\lambda}H_{-\omega}G_{-\mu-\epsilon}]w=[L_{k+1},G_{-\mu-\epsilon}]w,
$$
which implies that
$$
{\rm max~deg}([L_{k+1},L_{-\lambda}H_{-\omega}G_{-\mu-\epsilon}]w)\leq|\lambda+\omega+\mu+\epsilon|-k-1.
$$

Next, we assume $(\lambda,\omega)\neq(\bar{0},\bar{0})$ and write
$$
L_{-\lambda}H_{-\omega}=L_{-\lambda_s}\cdots L_{-\lambda_1}H_{-\omega_r}\cdots H_{-\omega_1},
$$
since $k<\omega_j, \lambda_i$ for all $1\leq i\leq s, 1\leq j\leq r,$ we have
$$
\begin{aligned}
{[L_{k+1},L_{-\lambda}H_{-\omega}G_{-\mu-\epsilon}]w}&{=[L_{k+1},L_{-\lambda}]H_{-\omega}G_{-\mu-\epsilon}w
+L_{-\lambda}[L_{k+1},H_{-\omega}]G_{-\mu-\epsilon}w}\\
&{+L_{-\lambda}H_{-\omega}[L_{k+1},G_{-\mu-\epsilon}]w\in U(\mathfrak{b}_{\epsilon}^-)w}
\end{aligned}
$$
and
$$
{\rm max~deg}([L_{k+1},L_{-\lambda}H_{-\omega}G_{-\mu-\epsilon}]w)=|\lambda+\omega+\mu+\epsilon|-k-1.
$$
Thus, (ii) follows.

For (iii), observe that
$$
\begin{aligned}
{[L_{k+1},L_{-\lambda}H_{-\omega}G_{-\mu-\epsilon}]w}&{=[L_{k+1},L_{-\lambda}]H_{-\omega}G_{-\mu-\epsilon}w +L_{-\lambda}[L_{k+1},H_{-\omega'}]H_{-k}^{\omega(k)}G_{-\mu-\epsilon}w}\\
&{+L_{-\lambda}H_{-\omega'}[L_{k+1},H_{-k}^{\omega(k)}]G_{-\mu-\epsilon}w+L_{-\lambda}H_{-\omega}[L_{k+1},G_{-\mu-\epsilon}]w}\\
\end{aligned}
$$
where $\omega{'}$ is an element in $\mathcal{P}$ such that $\omega{'}(k)=0,\omega{'}(i)=\omega(i)$ for  $i\neq k$. Based on the assumption about $k$, we see that $[L_{k+1},H_{-\omega{'}}]w,$ $[L_{k+1},L_{-\lambda}]w$, $[L_{k+1},G_{-\mu-\epsilon}]w$ belong to $U(\mathfrak{b}_{\epsilon}^{-})w$. Furthermore, we have $[L_{k+1},H_{-k}^{\omega(k)}]=-k\omega(k)H_{-k}^{\omega(k)-1}H_1$, and
$$
\begin{aligned}
{\rm max~deg}(L_{-\lambda}[L_{k+1},H_{-\omega'}]H_{-k}^{\omega(k)}G_{-\mu-\epsilon}w)\leq |\lambda+\omega+\mu+\epsilon|-k-1,\\
{\rm max~deg}([L_{k+1},L_{-\lambda}]H_{-\omega'}H_{-k}^{\omega(k)}G_{-\mu-\epsilon}w)\leq |\lambda+\omega+\mu+\epsilon|-k-1,\\
{\rm max~deg}(L_{-\lambda}H_{-\omega}H_{-k}^{\omega(k)}[L_{k+1},G_{-\mu-\epsilon}]w)\leq |\lambda+\omega+\mu+\epsilon|-k-1.
\end{aligned}
$$
Additionally, we have
$$
L_{-\lambda}H_{-\omega}[L_{k+1},H_{-k}^{\omega(k)}]G_{-\mu-\epsilon}w=
-k\omega(k)\psi(H_1)L_{-\lambda}H_{-\omega''}G_{-\mu-\epsilon}w,
$$
where $\omega''$ is an element in $\mathcal{P}$ such that $\omega''(k)=\omega(k)-1$ and $\omega''(i)=\omega(i)$ for $i\neq k$. Thus, (iii) holds.
\end{proof}

\begin{lemm} \label{L 3.3} The following statements hold.
\
\begin{itemize}
\item[(i)]For $m\in \mathbb{Z}_{+},$ we have
$$
{\rm max~deg}([H_m,L_{-\lambda}H_{-\omega}G_{-\mu-\epsilon}]w) \leq|\lambda+\omega+\mu+\epsilon|-m+1.
$$

\item[(ii)]Suppose
$\lambda=(0^{\lambda(0)},1^{\lambda(1)},2^{\lambda(2)},\cdots)$ and
$k$ is the minimal integer such that $\lambda(k)\neq0$ then
$$
[H_{k+1},L_{-\lambda}H_{-\omega}G_{-\mu-\epsilon}]w=
v-(k+1)\lambda(k)\psi(H_1)L_{-\lambda{'}}H_{-\omega}G_{-\mu-\epsilon}w,
$$
where if $k=0$ then $v=v{'}+v{''}$ with ${\rm max~deg}(v{'})<|\lambda+\omega+\mu+\epsilon|-k$ and
 $\max_{L_0}(v{''})<\lambda(k)-1,$ if $k>0$ then
${\rm max~deg}(v)<|\lambda+\omega+\mu+\epsilon|-k;$ ${\lambda^{'}}$
satisfies $\lambda{'}(k)=\lambda(k)-1, \lambda{'}(i)=\lambda(i)$
for all $i>k.$
\end{itemize}
\end{lemm}

\begin{proof}
To prove (i), we can write
$[H_{m},L_{-\lambda}H_{-\omega}G_{-\mu-\epsilon}]$ a linear combination of the basis of $U(\mathfrak{q}_{\epsilon})$,
$$
\begin{aligned}
{[H_m,L_{-\lambda}H_{-\omega}G_{-\mu-\epsilon}]}&{=\sum p_{{\lambda{'}},\omega^{'},\mu^{'},E_n}(C)L_{-\lambda'}H_{-\omega'}G_{-\mu'-\epsilon}E_{n}}\\
&{+\sum p_{\lambda^{''},\omega^{''},\mu^{''}}(C)L_{-\lambda^{''}}H_{-\omega^{''}}G_{-\mu^{''}-\epsilon}},
\end{aligned}
$$
where $\lambda{'}, \omega^{'}, \lambda^{''}, \omega^{''}$ are elements in $\mathcal{P}$, $\mu', \mu^{''}$ are elements in $\mathcal{Q}$ and $E_n$ can either be $H_n$ or $G_n$, where $n$ is a non-negative integer. These elements satisfy
$$
-|\lambda+\omega+\mu+\epsilon|+m=-|\lambda^{'}+\omega^{'}+\mu^{'}|+n=-|\lambda^{''}+\omega{''}+\mu^{''}|.
$$
Since $H_n w=\psi(H_n)w=0$ for $n\geq2$ and $G_{p-\epsilon} w=\psi(G_{p-\epsilon})w=0$ for $p\geq1$,then
$$
\begin{aligned}
{{\rm max~deg}([H_m,L_{-\lambda}H_{-\omega}G_{-\mu-\epsilon}]w)}&{={\rm max~deg}(\sum p_{{\lambda{'}},\omega^{'},\mu',E_n}(C)L_{-\lambda{'}}H_{-\omega'}G_{-\mu'-\epsilon}E_{n}w)}\\
 &{\leq|\lambda+\omega+\mu+\epsilon|-m+1}.
\end{aligned}
$$

For (ii), we denote $L_{-\lambda}=L_{-\lambda{'}}L_{-k}^{\lambda(k)},$ where $\lambda{'}(i)=\lambda(i)$ for all $i\neq k$ and $\lambda{'}(k)=0.$ Then we have
 $$
 \begin{aligned}
 {[H_{k+1},L_{-\lambda}H_{-\omega}G_{-\mu-\epsilon}]w}\quad~&{=[H_{k+1},L_{-\lambda{'}}]L_{-k}^{\lambda(k)}H_{-\omega}G_{-\mu-\epsilon}w}\\
 &{+L_{-\lambda{'}}[H_{k+1},L_{-k}^{\lambda(k)}]H_{-\omega}G_{-\mu-\epsilon}w}\\
 &{+L_{-\lambda{'}}L_{-k}^{\lambda(k)}[H_{k+1},H_{-\omega}]G_{-\mu-\epsilon}w}\\
 &{+L_{-\lambda{'}}L_{-k}^{\lambda(k)}H_{-\omega}[H_{k+1},G_{-\mu-\epsilon}]w}.
 \end{aligned}
 $$
By assumption about $k$, Definition \ref{D 2.2} and Lemma \ref{L 2.4}, we get that the degree of the first, the third and the last summands are strictly smaller than $|\lambda+\omega+\mu+\epsilon|-k$. Next, we deal with the second summand. If $k>0$, we have
$$
 \begin{aligned}
 {L_{-\lambda{'}}[H_{k+1},L_{-k}^{\lambda(k)}]H_{-\omega}G_{-\mu-\epsilon}w}&{=L_{-\lambda'}\sum_{i=0}^{\lambda(k)-1}L_{-k}^{i}[H_{k+1},L_{-k}]L_{-k}^{\lambda(k)-i-1}H_{-\omega}G_{-\mu-\epsilon}w}\\
 &{=-(k+1)L_{-\lambda'}\sum_{i=0}^{\lambda(k)-1}L_{-k}^{i}H_{1}L_{-k}^{\lambda(k)-i-1}H_{-\omega}G_{-\mu-\epsilon}w}\\
 &{=-\lambda(k)(k+1)\psi(H_1)L_{-\lambda'}L_{-k}^{\lambda(k)-1}H_{-\omega}G_{-\mu-\epsilon}w}\\
 &{+\sum p_{\bar{\lambda},\bar{\omega},\bar{\mu}}(C)L_{-\bar{\lambda}}H_{-\bar{\omega}}G_{-\bar{\mu}-\epsilon}w},
 \end{aligned}
 $$
 where
 $$
 {\rm max~deg}(\sum p_{\bar{\lambda},\bar{\omega},\bar{\mu}}(C)L_{-\bar{\lambda}}H_{-\bar{\omega}}G_{-\bar{\mu}-\epsilon})
 =|\lambda+\omega+\mu+\epsilon|-k-1.
 $$
 If $k=0$, we have
 $$
 \begin{aligned}
 {L_{-\lambda{'}}[H_1,L_{0}^{\lambda(0)}]H_{-\omega}G_{-\mu-\epsilon}w}&{=L_{-\lambda'}
 \sum_{i=0}^{\lambda(k)-1}L_{0}^i[H_1,L_{0}]L_{0}^{\lambda(0)-i-1}H_{-\omega}G_{-\mu-\epsilon}w}\\
 &{=-\lambda(0)\psi(H_1)L_{-\lambda{'}}L_{0}^{\lambda(0)-1}H_{-\omega}G_{-\mu-\epsilon}w}\\
 &{+\sum p_{\bar{\lambda},\bar{\omega},\bar{\mu}}(C)L_{-\lambda'}L_{-\bar{\lambda}}H_{-\bar{\omega}}G_{-\bar{\mu}-\epsilon}w}.
 \end{aligned}
 $$
 For each $p_{\bar{\lambda},\bar{\omega},\bar{\mu}}(C)\neq0$, $L_{\bar{\lambda}}=L_0^{l}$ for some $l<\lambda(k)-1$. Set
 $$
 \begin{aligned}
 {v'}&{=[H_{k+1},L_{-\lambda{'}}]L_{-k}^{\lambda(k)}H_{-\omega}G_{-\mu}w+L_{-\lambda{'}}L_{-k}^{\lambda(k)}[H_{k+1},H_{-\omega}]G_{-\mu}w}\\
 &{+L_{-\lambda{'}}L_{-k}^{\lambda(k)}H_{-\omega}[H_{k+1},G_{-\mu}]w},\\
 v''&{=\sum p_{\bar{\lambda},\bar{\omega},\bar{\mu}}(C)L_{-\lambda'}L_{-\bar{\lambda}}H_{-\bar{\omega}}G_{-\bar{\mu}-\epsilon}w}.
 \end{aligned}
 $$
 Thus (ii) holds.
\end{proof}

\begin{lemm}\label{L 3.5}
For any $(\lambda,\omega,\mu)\in \mathcal{P}\times\mathcal{P}\times\mathcal{Q}$, $k\in \mathbb{Z}_{+}$, we have
$$
{\rm max~deg}([G_{k-\epsilon},L_{-\lambda}H_{-\omega}G_{-\mu-\epsilon}]w)\leq |\lambda+\omega+\mu+\epsilon|-k+\epsilon.
$$
\end{lemm}

\begin{proof}
We write $[G_{k-\epsilon},L_{-\lambda}H_{-\omega}G_{-\mu-\epsilon}]$ a linear combination of the basis of $U(\mathfrak{q}_{\epsilon})$:
$$
\begin{aligned}
{[G_{k-\epsilon},L_{-\lambda}H_{-\omega}G_{-\mu-\epsilon}]}&{=\sum_{\lambda',\omega',\mu'}p_{\lambda',\omega',\mu'}L_{-\lambda'}H_{-\omega'}G_{-\mu'-\epsilon}}\\
&{+\sum_{\lambda'',\omega'',\mu,n}p_{\lambda'',\omega'',\mu,n}L_{-\lambda''}H_{-\omega''}G_{-\mu-\epsilon}G_{n-\epsilon}},
\end{aligned}
$$
where $p_{\lambda',\omega',\mu'}, p_{\lambda'',\omega'',\mu,n} \in \mathbb{C}$ and $\lambda',\lambda'',\omega',\omega''\in\mathcal{P},\mu'\in \mathcal{Q},n\in\mathbb{Z}_{+}$ satisfy
$$ -|\lambda'+\omega'+\mu'+\epsilon|=-|\lambda''+\omega''+\mu+\epsilon|+n-\epsilon=-|\lambda+\omega+\mu+\epsilon|+k-\epsilon.
$$
Since $G_iw=0$ for $i \in \mathbb{Z}_{+}-\epsilon$, we have
$$
\begin{cases}
[G_{k-\epsilon},L_{-\lambda}H_{-\omega}G_{-\mu-\epsilon}]w=0& if \quad~ k>|\lambda+\omega|;\\
\ \\
[G_{k-\epsilon},L_{-\lambda}H_{-\omega}G_{-\mu-\epsilon}]w=\sum_{\lambda',\omega',\mu'}L_{-\lambda'}H_{-\omega'}G_{-\mu'-\epsilon}w& if \quad~ 0<k\leq|\lambda+\omega|.
\end{cases}
$$
Then,  ${\rm max~deg}([G_k,L_{-\lambda}H_{-\omega}G_{-\mu-\epsilon}]w)\leq |\lambda+\omega+\mu+\epsilon|-k+\epsilon.$
\end{proof}

Now we are equipped to prove the main results of this section for the Whittaker vectors in the modules $M_{\psi}$ and $L_{\psi,\xi}$, applying Lemma \ref{L 3.3}, \ref{L 3.4} and \ref{L 3.5}.

\begin{theo}\label{P 3.6}
Suppose $M_{\epsilon}(\psi)$ is a universal Whittaker module for $V$, generated by the Whittaker vector $w= 1 \otimes 1 $.  Then $w' \in M_{\psi}$ is a Whittaker vector if and only if $w' \in {\rm span}_{S(C)}\{w,G_{-\epsilon}w\}$.
\end{theo}
\begin{proof}
Let $w' \in M_{\epsilon}(\psi)$ be an arbitrary vector. Then
$$
w' = \sum_{(\lambda, \omega, \mu) \in (\mathcal{P}, \mathcal{P},\mathcal{Q})} p_{\lambda, \omega, \mu}(C) L_{-\lambda}H_{-\omega}G_{-\mu}w
$$
for some polynomials $p_{\lambda, \omega, \mu}(C) \in S(C)$.

If $(\lambda,\omega,\mu)=(\overline{0},\overline{0},(0))$ and $p_{\lambda, \omega, \mu}(C)\neq0$, then $w'=p_{\lambda, \omega, \mu}(C)G_{-\epsilon}w$. By Definition \ref{D 2.2}, for any $m,n,r \in \mathbb{Z}_+$, we have
$$
\begin{aligned}
{(L_m-\psi(L_m)) w'}&{=p_{\lambda, \omega, \mu}(C)[L_m,G_{-\epsilon}]w=(-\epsilon-\frac{m}{2})G_{m-\epsilon}w=0,}\\
{(H_n-\psi(H_n)) w'}&{=p_{\lambda, \omega, \mu}(C)[H_n,G_{-\epsilon}]w=G_{n-\epsilon}w=0,}\\
{G_{r-\epsilon} w'}&{=p_{\lambda, \omega, \mu}(C)[G_{r-\epsilon},G_{-\epsilon}]w=0.}
\end{aligned}
$$
Hence, $w'=G_{-\epsilon}w$ is a Whittaker vector.

Next, we will prove that if there exist $\lambda\neq\overline{0},\omega\neq\overline{0}$ or $\mu \neq \overline{0}$ and $|\mu|>0$ such that $p_{\lambda, \omega, \mu}(C) \neq 0$, then there exists $m \in \mathbb{Z}_+$ such that $(L_m - \psi(L_m))w'  \neq 0$, $(H_m - \psi(H_m))w'\neq0$ or $ (G_{m-\epsilon} - \psi(G_{m-\epsilon}))w' \neq 0$. In this case $w'$ is not a Whittaker vector, that is, $w'$ is a Whittaker vector if and only if $w'=p(C)w,~p(C)\in S(C)$. This proves the result.

Let $N= {\rm max} \{ | \lambda+\omega+\mu| \mid  p_{\lambda, \omega, \mu}(C) \neq 0 \}$, and define
$$
\Lambda_N = \{ (\lambda, \omega, \mu) \in \mathcal{P}\times \mathcal{P}\times\mathcal{Q} \mid  p_{\lambda, \omega, \mu}(C) \neq 0, |\lambda+\omega+\mu| = N \}.
$$
If $N=0$, then $w'=p_{0^{m}, 0^{n}, 0}(C)L_0^{m}H_0^{n}G_{-\epsilon}w$. By Lemma \ref{L 3.1}, we have
$$
\begin{aligned}
{(H_1-\psi(H_1))w'}&{=p_{0^{m}, 0^{n},0}(C)[H_1,L_0^{m}H_0^{n}G_{-\epsilon}]w}\\
&{=-p_{0^{m}, 0^{n},0}(C)\sum_{i=0}^{m-1}L_0^iH_1L_0^{m-i-1}H_0^{n}G_{-\epsilon}w}\\
&{=m\psi(H_1)p_{0^{m-1}, 0^{n}, 0}(C)L_0^{m-1}H_0^{n}G_{-\epsilon}w}\\
&{+\sum_{\lambda',0^{m},0}p_{\lambda',0^{m},0}(C)L_{-\lambda'}H_0^{n}G_{-\epsilon}w},
\end{aligned}
$$
where $\lambda' \in \mathcal{P}$ satisfies $L_{\lambda'}=L_0^l$, $l<m-1$. Thus $(H_1-\psi(H_1))w'\neq0$, that is, $w'=p_{0^{\lambda(0)}, 0^{\omega(0)}, 0}(C)L_0^{\lambda(0)}H_0^{\omega(0)}G_{-\epsilon}w$ is not a Whittaker vector.

Now we suppose that $N>0$. Set
$$
k:=\text{ min$\{n\in \mathbb{N}|
\lambda(n)\neq0, \omega(n)\neq0$,\quad~or\quad~$\mu(n)\neq0$ \quad~for  \quad~ some \quad~ $(\lambda, \omega, \mu)\in \Lambda_N\}$}.
$$

{\bf{Case I :}} If $k$ satisfies $\lambda(k)\neq0$
for some $(\lambda, \omega, \mu)\in \Lambda_N$.

Using Lemma \ref{L 3.1}, we have
$$
\begin{aligned}
{(H_{k+1}-\psi(H_{k+1}))w{'}}&{=\sum_{(\lambda, \omega,\mu)\notin\Lambda_N}p_{\lambda, \omega,\mu}(C)[H_{k+1},L_{-\lambda}H_{-\omega}G_{-\mu-\epsilon}]w}\\
&{+\sum_{\begin{subarray}
\ \ (\lambda,\omega, \mu)\in\Lambda_N\\
\ \ \ \lambda(k)=0
\end{subarray}}p_{\lambda,\omega,\mu}(C)[H_{k+1},L_{-\lambda}H_{-\omega}G_{-\mu-\epsilon}]w}\\
&{+\sum_{\begin{subarray}
\ \ (\lambda,\omega, \mu)\in\Lambda_N\\
\ \ \ \lambda(k)\neq0
\end{subarray}}p_{\lambda,\omega,\mu}(C)[H_{k+1},L_{-\lambda}H_{-\omega}G_{-\mu-\epsilon}]w}
\end{aligned}
$$
If Lemma \ref{L 3.3}, part (i) is applied to the first sum, we find that its degree is strictly less than $N-k$. For the second sum, observe that $\lambda(i)=0$ for $0\leq i\leq k$,  then we have
$$
\begin{aligned}
{[H_{k+1},L_{-\lambda}H_{-\omega}G_{-\mu-\epsilon}]w}&{=[H_{k+1},L_{-\lambda}]H_{-\omega}G_{-\mu-\epsilon}w+L_{-\lambda}[H_{k+1},H_{-\omega}]G_{-\mu-\epsilon}w}\\
&{+L_{-\lambda}H_{-\omega}[H_{k+1},G_{-\mu-\epsilon}]w\in S(C)U(\mathfrak{q}_{\epsilon})w}.
\end{aligned}
$$
Hence its degree is also strictly smaller than $N+\epsilon-k$. Next, using Lemma \ref{L 3.3}, part (iii) to the third summand, we know it has the following form:
$$
v-\sum_{\begin{subarray}
\ \ (\lambda,\omega \mu)\in\Lambda_N\\
\ \ \lambda(k)\neq0
\end{subarray}}(k+1)\lambda(k)\psi(H_1)p_{\lambda',\omega, \mu}
(C)L_{-\lambda'}H_{-\omega}G_{-\mu-\epsilon}w,
$$
where if $k=0$ then $v=v{'}+v{''}$ such that ${\rm max~deg}(v{'})<N+\epsilon-k$ and $\max_{L_0}(v{''})<\lambda(0)-1;$ if $k>0$ then ${\rm max~deg}(v)<N+\epsilon-k$. Otherwise ${\lambda{'}}$ satisfies $\lambda{'}(k)=\lambda(k)-1$ and $\lambda{'}(i)=\lambda(i)$ for all $i>k.$ So the degree of the third summand is $N+\epsilon-k$. Hence, we conclude that
$(H_{k+1}-\psi(H_{k+1}))w{'}\neq0.$

{\bf{Case II :}} If $k$ satisfies $\omega(k)\neq0, \lambda(k)=0$ for some $(\lambda, \omega,\mu)\in \Lambda_N$.

To simplify, let us  denote  $w{'}=w{'}_1+w{'}_2$ such that
$$
w{'}_1=\sum_{(\lambda,\omega, \mu)\notin\Lambda_N}
p_{\lambda,\omega, \mu} (C) L_{-\lambda}H_{-\omega}G_{-\mu-\epsilon}w,
$$
$$
w{'}_2=\sum_{(\lambda,\omega, \mu)\in\Lambda_N}
p_{\lambda,\omega, \mu} (C) L_{-\lambda}H_{-\omega}G_{-\mu-\epsilon}w.
$$

Next, according to various subcases, we are looking for some element $E_m\in \mathfrak{q}_{\epsilon}^+$ such that $(E_m-\psi(E_m))w{'}\neq0$.

{\bf{Subcase 1 :}} ${\rm max~deg}(w{'}_1)<N+\epsilon-1$.

In this subcase,
$$
(L_{k+1}-\psi(L_{k+1}))w'_{1} =\sum_{(\lambda,\omega,\mu)\notin\Lambda_N} p_{\lambda,\omega, \mu} (C)[L_{k+1},L_{-\lambda}H_{-\omega}G_{-\mu-\epsilon}]w.
$$
Using the first inequality of Lemma \ref{L 3.4}, part (i), we have
$$
{\rm max~deg}((L_{k+1}-\psi(L_{k+1}))w'_{1})<(N+\epsilon-1)-(k+1)+2=N+\epsilon-k.
$$
Furthermore,
$$
\begin{aligned}
{(L_{k+1}-\psi(L_{k+1}))w'_{2}}&{= \sum_{\begin{subarray}
\ (\lambda,\omega, \mu)\in\Lambda_N\\
\ \ \omega(k)=0\end{subarray}}
p_{\lambda,\omega, \mu} (C)[L_{k+1},
L_{-\lambda}H_{-\omega}G_{-\mu-\epsilon}]w}\\
&{+\sum_{\begin{subarray}
\ (\lambda,\omega, \mu)\in\Lambda_N\\
\ \ \omega(k)\neq0\end{subarray}} p_{\lambda,\omega, \mu} (C)[L_{k+1},
L_{-\lambda}H_{-\omega}G_{-\mu-\epsilon}]w}.
\end{aligned}
$$

Applying Lemma \ref{L 3.4}, part (ii) to the first summand and part (iii) to the second summand, we have
$$
(L_{k+1}-\psi(L_{k+1}))w^{'}_{2}=v-\sum_{\begin{subarray}
\ (\lambda,\omega, \mu)\in\Lambda_N\\
\ \ \omega(k)\neq0\end{subarray}}
p_{\lambda,\omega, \mu} (C)\omega(k)k\psi(H_1)L_{-\lambda}H_{-\omega{'}}G_{-\mu-\epsilon}w,
$$
where ${\rm max~deg}(v)<N+\epsilon-k$ and $\omega{'}$ satisfies that $\omega{'}(i)=\omega(i)$
for all $i\neq k$ and $\omega{'}(k)=\omega(k)-1$, so ${\rm max~deg}((L_{k+1}-\psi(L_{k+1}))w^{'}_{2})=N+\epsilon-k$. As a result, we can conclude,
$$
(L_{k+1}-\psi(L_{k+1}))w{'}=
(L_{k+1}-\psi(L_{k+1}))w'_{1}+(L_{k+1}-\psi(L_{k+1}))w'_{2}\neq0.
$$

{\bf{Subcase 2 :}} ${\rm max~deg}(w{'}_1)=N+\epsilon-1$, and $\lambda=\overline{0}$ for any $(\lambda,\omega, \mu)\in \mathcal{P}\times\mathcal{P}\times\mathcal{Q}$ with $p_{\lambda, \omega,\mu}(C)\neq0$.

In this subcase,
$$
w'_1=\sum_{(\overline{0},\omega, \mu)\notin\Lambda_N} p_{\overline{0},\omega, \mu} (C) H_{-\omega}G_{-\mu-\epsilon}w,
$$
$$
w'_2=\sum_{(\overline{0},\omega, \mu)\in\Lambda_N} p_{\overline{0},\omega, \mu}(C)
H_{-\omega}G_{-\mu-\epsilon}w.
$$
The quantity $(L_{k+1}-\psi(L_{k+1}))w^{'}_1$ can be obtained as,
$$
(L_{k+1}-\psi(L_{k+1}))w'_{1}=\sum_{(\overline{0},\omega, \mu)\notin\Lambda_N}
p_{\overline{0},\omega, \mu}(C)[L_{k+1},H_{-\omega}G_{-\mu-\epsilon}]w.
$$
By using the second inequality of Lemma \ref{L 3.4}, part (i), we have
$$
{\rm max~deg}((L_{k+1}-\psi(L_{k+1}))w'_{1})\leq N+\epsilon-1-(k+1)+1=N+\epsilon-k-1.
$$
Similarly, for $(L_{k+1}-\psi(L_{k+1}))w^{'}_2$, we have,
$$
\begin{aligned}
{(L_{k+1}-\psi(L_{k+1}))w^{'}_{2}} &{~=\sum_{\begin{subarray} \
(\overline{0},\omega, \mu)\in\Lambda_N\\
\ \ \omega(k)=0
\end{subarray}}
p_{\overline{0},\omega, \mu}(C)[L_{k+1},H_{-\omega}G_{-\mu-\epsilon}]w}\\
&{+\sum_{\begin{subarray} \
(\overline{0},\omega, \mu)\in\Lambda_N\\
\ \ \omega(k)\neq0
\end{subarray}}
p_{\overline{0},\omega, \mu}(C)[L_{k+1},H_{-\omega}G_{-\mu-\epsilon}]w}.
\end{aligned}
$$
Applying Lemma \ref{L 3.4}, part (ii) to the first sum and part (iii) to the second sum, we obtain
$$
(L_{k+1}-\psi(L_{k+1}))w'_{2}=v-\sum_{\begin{subarray} \
(\overline{0},\omega, \mu)\in\Lambda_N\\
\ \ \omega(k)\neq0
\end{subarray}}
p_{\overline{0},\omega, \mu}(C)\omega(k)k\psi(H_1)H_{-\omega{'}}G_{-\mu-\epsilon}w,
$$
where ${\rm max~deg}(v)<N+\epsilon-k$ and $\omega{'}$ satisfies $\omega{'}(i)$$=\omega(i)$ for all $i\neq k$ and $\omega{'}(k)=\omega(k)-1$. Thus we deduce that ${\rm max~deg}((L_{k+1}-\psi(L_{k+1}))w)=N+\epsilon-k$, which implies that $(L_{k+1}-\psi(L_{k+1}))w\neq0$.

{\bf{Subcase 3 :}} ${\rm max~deg}(w'_1)=N+\epsilon-1$, $\lambda=\bar{0}$ for any $(\lambda,\omega,\mu)\in\Lambda_N$ and $\lambda\neq\bar{0}$ for some $(\lambda,\omega,\mu)$ such that $(\lambda,\omega,\mu)\notin\Lambda_N$ with $p_{\lambda,\omega,\mu}(C)\neq0.$

In this subcase,
$$
\begin{aligned}
w'_1&{=\sum_{(\lambda,\omega, \mu)\notin\Lambda_N} p_{\lambda,\omega, \mu} (C) L_{-\lambda}H_{-\omega}G_{-\mu-\epsilon}w},\\
w'_2&{=\sum_{(\overline{0},\omega, \mu)\in\Lambda_N} p_{\overline{0},\omega, \mu}(C)
H_{-\omega}G_{-\mu-\epsilon}w.}
\end{aligned}
$$
It follows that,
$$
(G_{k-\epsilon}-\psi(G_{k-\epsilon}))w'_{1} =\sum_{(\lambda,\omega,\mu)\notin\Lambda_N} p_{\lambda,\omega, \mu} (C)[G_{k-\epsilon},L_{-\lambda}H_{-\omega}G_{-\mu-\epsilon}]w.
$$
Applying Lemma \ref{L 3.5}, we deduce,
$$
{\rm max~deg}((G_{k-\epsilon}-\psi(G_{k-\epsilon}))w'_{1})\leq N+2\epsilon-k-1.
$$
In addition, we can express ${(G_{k-\epsilon}-\psi(G_{k-\epsilon}))w^{'}_{2}}$ as
$$
\begin{aligned}
{(G_{k-\epsilon}-\psi(G_{k-\epsilon}))w^{'}_{2}} &{~=\sum_{\begin{subarray} \
(\bar{0},\omega, \mu)\in\Lambda_N\\
\ \ \omega(k)=0
\end{subarray}}
p_{\bar{0},\omega, \mu}(C)[G_{k-\epsilon},H_{-\omega}G_{-\mu-\epsilon}]w}\\
&{+\sum_{\begin{subarray} \
(\bar{0},\omega, \mu)\in\Lambda_N\\
\ \ \omega(k)\neq0
\end{subarray}}
p_{\bar{0},\omega, \mu}(C)[G_{k-\epsilon},H_{-\omega}G_{-\mu-\epsilon}]w},
\end{aligned}
$$
By referring to Lemma \ref{L 3.5}, we find,
$$
{\rm max~deg}((G_{k-\epsilon}-\psi(G_{k-\epsilon}))w'_{2})= N+2\epsilon-k.
$$
Consequently, we can conclude that $(G_{k-\epsilon}-\psi(G_{k-\epsilon}))w'\neq0$.

{\bf{Subcase 4 :}} ${\rm max~deg}(w'_1)=N+\epsilon-1$ and there exists a $(\lambda,\omega, \mu)\in\Lambda_N$ such that $\lambda\neq\bar{0}$.

In this subcase, we assume that $l$ is minimal such that $\lambda(l)\neq0$ for some $(\lambda,\omega, \mu)\in\Lambda_{N}$. Given the assumption about $k$, we have $l>k$. Since
$$
(H_{l+1}-\psi(H_{l+1}))w^{'}_{1} =\sum_{(\lambda,\omega,\mu)\notin\Lambda_N} p_{\lambda,\omega, \mu} (C)[H_{l+1},L_{-\lambda}H_{-\omega}G_{-\mu-\epsilon}]w,
$$

Applying Lemma \ref{L 3.3}, part (i), we obtain
$$
{\rm max~deg}((H_{l+1}-\psi(H_{l+1}))w'_{1})\leq N+\epsilon-1-(l+1)+1= N+\epsilon-l-1.
$$
In addition, Using Lemma \ref{L 3.3}, part (iii), we have
$$
\begin{aligned}
{(H_{l+1}-\psi(H_{l+1}))w'_{2}}~&{=\sum_{(\lambda,\omega, \mu)\in\Lambda_N}
p_{\lambda,\omega, \mu} (C)[H_{l+1},
L_{-\lambda}H_{-\omega}G_{-\mu-\epsilon}]w}\\
&{=v-\sum_{\begin{subarray}
\ (\lambda,\omega, \mu)\in\Lambda_N\\
\ \  \lambda(l)\neq0\end{subarray}}
p_{\lambda,\omega, \mu}(C)(l+1)\psi(H_1)\lambda(l)
L_{-\lambda'}H_{-\omega}G_{-\mu-\epsilon}w},
\end{aligned}
$$
where ${\rm max~deg}(v)<N-l$ and $ \lambda'$ satisfies that $\lambda'(i)=\lambda(i)$ for all $i\neq l$ and $\lambda'(l)=\lambda(l)-1$. Hence, we have
$$
{\rm max~deg}((H_{l+1}-\psi(H_{l+1}))w'_2)=N+\epsilon-l,
$$
which implies that
$$
(H_{l+1}-\psi(H_{l+1}))w'=
(H_{l+1}-\psi(H_{l+1}))w'_{1}+(H_{l+1}-\psi(H_{l+1}))w'_{2}\neq0.
$$

{\bf{Case \uppercase\expandafter{\romannumeral3} :}} $k$ satisfies $\mu(k)\neq0, \lambda(k)=\omega(k)=0$ for some $(\lambda, \omega,\mu)\in \Lambda_N$.

{\bf{Subcase 1 :}} $\lambda=\omega=\overline{0}$ for any $(\lambda,\omega, \mu)\in \mathcal{P}\times\mathcal{P}\times\mathcal{Q}$ with $p_{\lambda, \omega,\mu}(C)\neq0$.

In this subcase,
$$
w'=\sum_{(\overline{0},\overline{0},\mu)}p_{\overline{0},\overline{0},\mu}(C)G_{-\mu-\epsilon}w.
$$
Since $N>0$, using Lemma \ref{L 3.1}, we have
$$
(H_1-\psi(H_1)w'=\sum_{(\overline{0},\overline{0},\mu)}p_{\overline{0},\overline{0},\mu}(C)[H_1,G_{-\mu-\epsilon}]w\in U(\mathfrak{b_{\epsilon}^-}).
$$
Hence, $(H_1-\psi(H_1)w'\neq0$, that is, $w'$ is not a Whittaker vector.

{\bf{Subcase 2 :}} $\lambda\neq\overline{0}$ or $\omega\neq\overline{0}$ for some $(\lambda, \omega,\mu)\in \Lambda_N$.

In this subcase,
$$
\begin{aligned}
{G_{k-\epsilon}w'}~&{=\sum_{(\lambda,\omega,\mu)\in \Lambda_N}p_{\lambda,\omega,\mu}(C)[G_{k-\epsilon},L_{-\lambda}H_{-\omega}G_{-\mu-\epsilon}]w}\\
&{+\sum_{(\lambda,\omega,\mu)\notin \Lambda_N}p_{\lambda,\omega,\mu}(C)[G_{k-\epsilon},L_{-\lambda}H_{-\omega}G_{-\mu-\epsilon}]w}.
\end{aligned}
$$
By Lemma \ref{L 3.5} and the assumption about $k$, we have
$$
{\rm max~deg}(\sum_{(\lambda,\omega,\mu)\notin \Lambda_N}p_{\lambda,\omega,\mu}(C)[G_{k-\epsilon},L_{-\lambda}H_{-\omega}G_{-\mu-\epsilon}]w)<N+\epsilon-k;
$$
$$
{\rm max~deg}(\sum_{(\lambda,\omega,\mu)\in \Lambda_N}p_{\lambda,\omega,\mu}(C)[G_{k-\epsilon},L_{-\lambda}H_{-\omega}G_{-\mu-\epsilon}]w)=N+\epsilon-k.
$$
Hence, we conclude that $G_{k-\epsilon}w'\neq0$.

{\bf{Subcase 3 :}} $\lambda=\omega=\overline{0}$ for any $(\lambda, \omega,\mu)\in \Lambda_N$ and $\lambda\neq\overline{0}$ or $\omega\neq\overline{0}$ for some $(\lambda, \omega,\mu)\notin \Lambda_N$ with $p_{\lambda, \omega,\mu}(C)\neq0$.

In this subcase,
$$
w'=\sum_{(\overline{0},\overline{0},\mu)\in \Lambda_N}p_{\overline{0},\overline{0},\mu}(C)G_{-\mu-\epsilon}w+\sum_{(\lambda,\omega,\mu)\notin \Lambda_N}p_{\lambda,\omega,\mu}(C)L_{-\lambda}H_{-\omega}G_{-\mu-\epsilon}w.
$$
If $w'=\sum_{\mu}p_{\mu}(C)L_0^mH_0^nG_{-\mu-\epsilon}w$, then
$$
\begin{aligned}
{(H_1-\psi(H_1))w'}&{=\sum_{\mu}p_{\mu}(C)[H_1,L_0^mH_0^nG_{-\mu-\epsilon}]w}\\
&{=\sum_{\lambda,\mu'}p_{\lambda,\mu'}(C)L_{-\lambda}H_0^nG_{-\mu'-\epsilon}w}\\
&{-\sum_{\mu}p_{\mu}(C)m\psi(H_1)L_0^{m-1}H_0^nG_{-\mu-\epsilon}w}.
\end{aligned}
$$
Thus, $(H_1-\psi(H_1))w'\neq0$, that is, $w'=\sum_{\mu}p_{\mu}(C)L_0^mH_0^nG_{-\mu-\epsilon}w$ is not a Whittaker vector.

We set
$$
l={\rm min}\{i~|~i>0~{\rm such~that}~\lambda(i)\neq0~{\rm or}~\omega(i)\neq0~{\rm for}~(\lambda,\omega,\mu) \notin \Lambda_N\}.
$$
Using Definition \ref{D 2.2} and Lemma \ref{L 3.1}, we have
$$
\begin{aligned}
G_{l-\epsilon}w'&{=\sum_{(\overline{0},\overline{0},\mu)\in \Lambda_N}p_{\overline{0},\overline{0},\mu}(C)[G_{l-\epsilon},G_{-\mu-\epsilon}]w}\\
&{+\sum_{(\lambda,\omega,\mu)\notin \Lambda_N}p_{\lambda,\omega,\mu}(C)[G_{l-\epsilon},L_{-\lambda}H_{-\omega}G_{-\mu-\epsilon}]w}\\
&{=\sum_{\lambda',\omega',\mu'}p_{\lambda',\omega',\mu'}L_{-\lambda'}H_{-\omega'}G_{-\mu'-\epsilon}w \in U(\mathfrak{b_{\epsilon}^-})}.
\end{aligned}
$$
Hence, $G_{l-\epsilon}w'\neq0$.
\end{proof}

\begin{coro}\label{C 3.6}
The center of $U(\mathfrak{q}_{\epsilon})$ is $S(C)$.
\end{coro}

\begin{proof}
For any $u \in Z(U(\mathfrak{q}_{\epsilon}))$, $uw$ is a Whittaker vector.  By applying Theorem \ref{P 3.6}, we conclude that $u \in S(C)$. This implies that $Z(U(\mathfrak{q}_{\epsilon}))\subseteq S(C)$. On the other hand, clearly $S(C) \subseteq Z(U(\mathfrak{q}_{\epsilon}))$. Therefore, we deduce that $Z(U(\mathfrak{q}_{\epsilon}))= S(C)$.
\end{proof}

\begin{theo}\label{P 3.7}
Let $w = 1 \otimes 1 \in M_{\epsilon}(\psi)$ and $\overline{w}= \overline{1 \otimes 1} \in L_{\epsilon}(\psi, \xi)$. Then $\overline{w}' \in L_{\epsilon}(\psi, \xi)$ is a Whittaker vector if and only if $\overline{w}' \in {\rm span}_{\mathbb{C}}\{\overline{w},G_{-\epsilon}\overline{w}\}$.
\end{theo}

\begin{proof}
Note that the set $\{L_{- \lambda}H_{-\omega}G_{-\mu-\epsilon} \overline{w} \mid (
\lambda,\omega, \mu) \in \mathcal{P}\times\mathcal{P}\times\mathcal{Q}\}$ spans $L_{\epsilon}(\psi, \xi)$. We assert that this set is linearly independent and therefore a basis for $L_{\epsilon}(\psi, \xi)$. To verify this, assume there exist  $p_{\lambda,\omega,\mu} \in \mathbb{C}$, with at most finitely many $p_{\lambda,\omega,\mu} \neq 0$, such that
$$
0=\sum_{\lambda,\omega,\mu} p_{\lambda,\omega,\mu}L_{-\lambda}H_{-\omega}G_{-\mu-\epsilon} \overline{w}= \overline{ \sum_{\lambda,\omega,\mu} p_{\lambda,\omega,\mu} L_{-\lambda}H_{-\omega}G_{-\mu-\epsilon}w}
$$
in $L_{\epsilon}(\psi, \xi)$, that is to say, $\sum_{\lambda,\omega,\mu} p_{\lambda,\omega,\mu} L_{-\lambda}H_{-\omega}G_{-\mu-\epsilon}w \in  (C - \xi)M_{\epsilon}(\psi)$. Hence
$$
\sum_{\lambda,\omega,\mu} p_{\lambda,\omega,\mu} L_{-\lambda}H_{-\omega}G_{-\mu-\epsilon}w = (C - \xi) \sum_{\begin{subarray}
\ \ \lambda,\omega,\mu\\
0 \leq i \leq k
\end{subarray}} q_{\lambda,\omega,\mu,i} C^i L_{-\lambda}H_{-\omega}G_{-\mu-\epsilon} w
$$
for some $k \in \mathbb{Z}^+$ and $q_{\lambda,\omega,\mu,i} \in \mathbb{C}$. This expression can be rewritten as
$$
\begin{aligned}
0~&{=\sum_{\lambda,\omega,\mu} (p_{\lambda,\omega,\mu} + \xi q_{\lambda,\omega,\mu,0}) L_{-\lambda}H_{-\omega}G_{-\mu-\epsilon}w}\\
 &{+ \sum_{\begin{subarray}
\ \ \lambda,\omega,\mu\\
 1 \leq i \leq k
 \end{subarray}} (\xi q_{\lambda,\omega,\mu,i} - q_{\lambda,\omega,\mu,i-1}) C^i L_{-\lambda}H_{-\omega}G_{-\mu-\epsilon}w}\\
&{-\sum_{\lambda,\omega,\mu} q_{\lambda,\omega,\mu, k} C^{k+1} L_{-\lambda}H_{-\omega}G_{-\mu-\epsilon} w}.
\end{aligned}
$$
Due to the fact that each of the basis vectors (in $M_{\epsilon}(\psi)$) in this linear combination are distinct, we conclude that
$$
q_{\lambda,\omega,\mu, k}=0,~\xi q_{\lambda,\omega,\mu, i} - q_{\lambda,\omega,\mu, i-1}=0~{\rm and}~p_{\lambda,\omega,\mu} - \xi q_{\lambda,\omega,\mu,0}=0,
$$
and thus $p_{\lambda,\omega,\mu}=0$ for all $(\lambda,\omega,\mu) \in \mathcal{P}\times\mathcal{P}\times\mathcal{Q}$.

With this fact now proved, it is possible to use the same argument as in Theorem \ref{P 3.6} to complete the proof. In fact, 
we simply replace the polynomials $p_{\lambda,\omega,\mu} (C)$ in $C$ with scalars $p_{\lambda,\omega,\mu}$ whenever necessary.
\end{proof}

Since $\overline{w}\notin \langle G_{-\epsilon}\overline{w}\rangle$, then the Whittaker module generated by $G_{-\epsilon}\overline{w}$ is a proper submodule of $L_{\epsilon}(\psi,\xi)$. The vector $G_{-\epsilon}\overline{w}$ is called the degenerate Whittaker vector. Let $\widetilde{{L_{\epsilon}(\psi,\xi)}}=L_{\epsilon}(\psi,\xi)/\langle G_{-\epsilon}\overline{w}\rangle$, then $\widetilde{{L_{\epsilon}(\psi,\xi)}}$ is a Whittaker module of type $\psi$ with cyclic Whittaker vector $\widetilde{w}=\widetilde{1\otimes1}$.

\section{Simple Whittaker modules}
In this section, we prove the simplicity of the modules $\widetilde{L_{\epsilon}(\psi,\xi)}$ and establish that they form a complete set of simple Whittaker modules, up to isomorphism. Additionally, we will provide a criterion for the irreducibility of Whittaker modules.

Let us consider a fixed Lie superalgebra homomorphism $\psi: \mathfrak{q}_{\epsilon}^+ \rightarrow \mathbb{C}$ and a Whittaker module $V$ of type $\psi$. We can view  $V$ as a $\mathfrak{q}_{\epsilon}^+$-module by restricting the action of $\mathfrak{q}_{\epsilon}$ on $V$. To further define the action of $\mathfrak{q}_{\epsilon}^+$ on $V$, we introduce a new operation known as the dot action, denoted by the symbol $\cdot$. This dot action operates on elements of $\mathfrak{q}_{\epsilon}^+$ and $V$ according to the rule
$$
x \cdot v = xv - \psi (x) v,~{\rm for}~x \in \mathfrak{q}_{\epsilon}^+~{\rm and}~v \in V.
$$
Therefore, if we consider a Whittaker module $V$ as a $\mathfrak{q}_{\epsilon}^+$-module under the dot action, it follows that $E_n \cdot v = E_nv - \psi(E_n)v=[E_n,u]w$ for $E_n=L_n,H_n$, $n\in \mathbb{Z}_+$ or $E_n=G_{n}$, $n \in \mathbb{Z}_{+}-\epsilon$ and
$v=uw \in V$.

\begin{lemm}\label{L 4.2}
If $n>0$, then $L_n$, $H_n$ and $G_{n-\epsilon}$ are locally nilpotent on $V$ under the dot action.
\end{lemm}

\begin{proof}
Let $E_n=L_n,H_n$, $n\in \mathbb{Z}_+$ or $E_n=G_n$, $n \in \mathbb{Z}_{+}-\epsilon$. Since
$$
V = {\rm span}_{S(C)} \{L_{-\lambda}H_{-\omega}G_{-\mu-\epsilon}w \mid (\lambda,\omega,\mu) \in \mathcal{P}\times\mathcal{P}\times\mathcal{Q}\},
$$
it is sufficient to show that some power of $(E_n - \psi(E_n)1)$ annihilates $L_{-\lambda}H_{-\omega}G_{-\mu-\epsilon}w$.
We observe that
$$
(E_n-\psi(E_n)1)^k L_{-\lambda}H_{-\omega}G_{-\mu-\epsilon}w={\rm ad}_{E_n}^k(L_{-\lambda}H_{-\omega}G_{-\mu-\epsilon})w
$$
and ${\rm ad}_{E_n}^k(L_{-\lambda}H_{-\omega}G_{-\mu-\epsilon})\in U(\mathfrak{q}_{\epsilon})_{-|\lambda+\omega+\mu+\epsilon| + nk}$. Furthermore, ${\rm ad}_{E_n}^k(L_{-\lambda}H_{-\omega}G_{-\mu-\epsilon})$ can be expressed as a sum of terms of the form $E_{a_1}\cdots E_{a_s}$, where $s\leq l(\lambda)+l(\omega)+l(\mu)$ and $a_1\leq \cdots \leq a_s$. Now $E_{a_1}\cdots E_{a_s}w=0$ for $a_s>2$. If $a_1\leq \cdots \leq a_s \leq2$, then $E_{a_1}\cdots E_{a_s}\in U(\mathfrak{q}_{\epsilon})_m$, where $m=a_1+\cdots+a_s\leq2s$. Therefore, by choosing a $k$ sufficiently large such that
$$
2 \cdot (l(\lambda)+l(\omega)+l(\mu)) < -|\lambda+\omega+\mu+\epsilon| + nk,
$$
we ensure that $a_s>2$, and thus $E_{a_1}\cdots E_{a_s}w=0$. This finishes the proof.
\end{proof}

\begin{lemm}\label{L 4.3}
Let $\lambda,\omega \in \mathcal{P},\mu \in \mathcal{Q},i>0$ and $E_n=L_n,H_n$, , $n \in \mathbb{Z}_{+}$ or $E_n=G_n$, $n \in \mathbb{Z}_{+}-\epsilon$.
\begin{itemize}
\item[(i)] For all $n>0$, $E_n \cdot (C^i L_{-\lambda}H_{-\omega}G_{-\mu-\epsilon} w) \in {\rm span}_{\mathbb{C}} \{ C^j L_{-\lambda'}H_{-\omega'}G_{-\mu'-\epsilon} w \mid |\lambda'+\omega'+\mu'+\epsilon| + \lambda'(0) \leq |\lambda+\omega+\mu+\epsilon| + \lambda(0), j=i,i+1 \}$.
\item[(ii)]  If $n > |\lambda+\omega+\mu+\epsilon|+2$, then $E_n \cdot (L_{-\lambda}H_{-\omega}G_{-\mu-\epsilon}w) = 0$.
\end{itemize}
\end{lemm}

\begin{proof}
(i)  Let $E_n = L_n, H_n$, or $G_n$. Consider the action of $E_n$ on $C^i L_{-\lambda}H_{-\omega}G_{-\mu} w$ under the dot action. Using the definition of the dot action, we have,
$$
E_n \cdot (C^i L_{-\lambda}H_{-\omega}G_{-\mu-\epsilon} w)=C^i(E_n \cdot (L_{-\lambda}H_{-\omega}G_{-\mu-\epsilon} w)),
$$
it is sufficient to show that (i) holds for $i=0$. The result for $l(\lambda)+l(\omega)+l(\mu)=0$ is obvious. Now we prove the result for $l(\lambda)+l(\omega)+l(\mu)>0$.

{\bf{Case I :}} $E_n=G_n$, $n \in \mathbb{Z}_{+}-\epsilon$.

In this case, ${G_n \cdot (L_{-\lambda}H_{-\omega}G_{-\mu-\epsilon}w)}=[G_n,L_{-\lambda}H_{-\omega}G_{-\mu'-\epsilon}]w$. By similar calculation in Lemma \ref{L 3.5}, if $n>|\lambda+\omega|$, then $[G_n,L_{-\lambda}H_{-\omega}G_{-\mu'-\epsilon}]w=0$, and if $0<n\leq|\lambda+\omega|$, we have
$$
[G_n,L_{-\lambda}H_{-\omega}G_{-\mu'-\epsilon}]w=\sum_{\lambda',\omega',\mu}p_{\lambda',\omega',\mu'}L_{-\lambda'}H_{-\omega'}G_{-\mu'-\epsilon}w,
$$
where $p_{\lambda',\omega',\mu'} \in \mathbb{C}$ and $\lambda',\omega'\in\mathcal{P}$ satisfy $|\lambda'+\omega'+\mu+\epsilon|=|\lambda+\omega+\mu+\epsilon|-n$ and $\lambda'(0)<\lambda(0)$.

{\bf{Case II :}} $E_n=H_n$, $n \in \mathbb{Z}_+$.

In this case, using similar calculation in Lemma \ref{L 3.3}, we have
$$
\begin{aligned}
{H_n \cdot (L_{-\lambda}H_{-\omega}G_{-\mu-\epsilon}w)}&{=[H_n,L_{-\lambda}H_{-\omega}G_{-\mu-\epsilon}]w}
=\sum p_{{\lambda'},\omega',\mu'}(C)L_{-\lambda'}H_{-\omega'}G_{-\mu'-\epsilon}w,
\end{aligned}
$$
where deg($p_{{\lambda'},\omega',\mu'}(C)$)= 0 or 1, $\lambda', \omega' \in \mathcal{P}$ and $\mu' \in \mathcal{Q}$  satisfy
$$
|\lambda+\omega+\mu+\epsilon|-n+1 \geq |\lambda'+\omega'+\mu'+\epsilon|,
$$
and $\lambda'(0)<\lambda(0)$. Since $n \geq 1$, then $|\lambda'+\omega'+\mu'+\epsilon|+\lambda'(0)<|\lambda+\omega+\mu+\epsilon|+\lambda(0)$.

{\bf{Case III :}} $E_n=L_n$, $n \in \mathbb{Z}_+$.

If $\lambda=\bar{0}$ the result follows from Lemma \ref{L 3.4}. If $|\lambda|=0$, using similar calculation in Lemma \ref{L 3.3}, we have
$$
\begin{aligned}
{L_n \cdot (L_{0}^{m}H_{-\omega}G_{-\mu-\epsilon}w)}&{=[L_n,L_{0}^{m}H_{-\omega}G_{-\mu-\epsilon}]w}\\
&{=\sum_{\lambda',\omega',\mu'}p_{{\lambda{'}},\omega^{'},\mu^{'}}(C)L_{-\lambda'}H_{-\omega'}G_{-\mu'-\epsilon}w}\\
&{+\sum_{\lambda'',\omega'',\mu''}p_{\lambda'',\omega'',\mu''}(C)L_{-\lambda''}H_{-\omega''}G_{-\mu''-\epsilon}L_nw},
\end{aligned}
$$
where deg($p_{{\lambda{'}},\omega^{'},\mu^{'}}(C)$)= 0 or 1, deg($p_{{\lambda{''}},\omega^{''},\mu^{''},k}(C)$)= 0 or 1, $|\lambda'+\omega'+\mu'+\epsilon|\leq|\lambda+\omega+\mu+\epsilon|-n+1$, $\lambda'(0)\leq\lambda(0)$ and $|\lambda''+\omega''+\mu''+\epsilon|=|\lambda+\omega+\mu+\epsilon|$, $\lambda''(0)\leq\lambda(0)-1$. Thus
$$
|\lambda'+\omega'+\mu'+\epsilon|+\lambda'(0)\leq|\lambda+\omega+\mu+\epsilon|+\lambda(0).
$$
$$
|\lambda''+\omega''+\mu''+\epsilon|+\lambda''(0)\leq|\lambda+\omega+\mu+\epsilon|+\lambda(0).
$$

Now, we suppose that $|\lambda|>0$ and prove this result by induction on $|\lambda+\omega+\mu+\epsilon|+\lambda(0)$. Let $m={\rm max}\{i \mid \lambda(i)>0\}$, and define $L_{-\lambda}=L_{-m}L_{-\lambda'}$. Then
$$
\begin{aligned}
{L_n \cdot (L_{-\lambda}H_{-\omega}G_{-\mu-\epsilon}w)}&{=[L_n,L_{-m}]L_{-\lambda'}H_{-\omega}G_{-\mu-\epsilon}w+L_{-m}[L_n,L_{-\lambda'}H_{-\omega}G_{-\mu-\epsilon}]w}\\
&{=-((m+n)L_{n-m}+\frac{n^3-n}{12}\delta_{m,n}C)L_{-\lambda'}H_{-\omega}G_{-\mu-\epsilon}w}\\
&{+L_{-m}[L_n,L_{-\lambda'}H_{-\omega}G_{-\mu-\epsilon}]w},
\end{aligned}
$$
where $|\lambda'|=|\lambda|-m$ and $\lambda'(0)=\lambda(0)$. So
$$
|\lambda'+\omega+\mu+\epsilon|+\lambda'(0)<|\lambda+\omega+\mu+\epsilon|+\lambda(0).
$$
By induction, $[L_n,L_{-\lambda'}H_{-\omega}G_{-\mu-\epsilon}]w \in {\rm span}_{\mathbb{C}}\{C^jL_{-\bar{\lambda}}H_{-\bar{\omega}}G_{-\bar{\mu}-\epsilon}w \mid |\bar{\lambda}+\bar{\omega}+\bar{\mu}+\epsilon|+\bar{\lambda}(0)\leq|\lambda'+\omega+\mu+\epsilon|+\lambda'(0),j=0,1\}$. It then follows that $L_{-m}[L_n,L_{-\lambda'}H_{-\omega}G_{-\mu-\epsilon}]w \in {\rm span}_{\mathbb{C}}\{C^jL_{-\bar{\lambda}}H_{-\bar{\omega}}G_{-\bar{\mu}-\epsilon}w \mid |\bar{\lambda}+\bar{\omega}+\bar{\mu}+\epsilon|+\bar{\lambda}(0)\leq|\lambda+\omega+\mu+\epsilon|+\lambda(0),j=0,1\}$, because $-m<0, |\lambda'|=|\lambda|-m$ and $\lambda'(0)=\lambda(0)$.

To see why $L_{n-m}L_{-\lambda'}H_{-\omega}G_{-\mu-\epsilon}w$ has the desired form, we consider two cases. If $n-m<0$, rewrite $L_{n-m}L_{-\lambda'}H_{-\omega}G_{-\mu-\epsilon}w$ in a linear combination of the basis of $M_{\epsilon}(\psi)$:
$$
L_{n-m}L_{-\lambda'}H_{-\omega}G_{-\mu-\epsilon}w=\sum_{\lambda''}p_{\lambda''}(C)L_{-\lambda''}H_{-\omega}G_{-\mu-\epsilon}w,
$$
where deg$(p_{\lambda''}(C))=0$ or 1,
$$
|\lambda''+\omega+\mu+\epsilon|=|\lambda'+\omega+\mu+\epsilon|+m-n=|\lambda+\omega+\mu+\epsilon|-n
$$
and $\lambda''(0)=\lambda'(0)=\lambda(0)$, then
$$
|\lambda''+\omega+\mu+\epsilon|+\lambda''(0)<|\lambda+\omega+\mu+\epsilon|+\lambda(0).
$$
If $n-m>0$, note that
$$
L_{n-m}L_{-\lambda'}H_{-\omega}G_{-\mu-\epsilon}w=\psi(L_{n-m})L_{-\lambda'}H_{-\omega}G_{-\mu-\epsilon}w+L_{n-m} \cdot (L_{-\lambda'}H_{-\omega}G_{-\mu-\epsilon}w).
$$
It follows from induction that $L_{n-m} \cdot (L_{-\lambda'}H_{-\omega}G_{-\mu-\epsilon}w)$ has the desired form, and therefore $L_{n-m}L_{-\lambda'}H_{-\omega}G_{-\mu-\epsilon}w$ has the correct form.

With respect to (ii), we have
$$
\begin{aligned}
{E_n \cdot (L_{-\lambda}H_{-\omega}G_{-\mu-\epsilon}w)}
&{=E_nL_{-\lambda}H_{-\omega}G_{-\mu-\epsilon}w-\psi(E_n)L_{-\lambda}H_{-\omega}G_{-\mu-\epsilon}w}\\
&{=\sum_{\lambda',\omega',\mu',s}p_{\lambda',\omega',\mu',s}L_{-\lambda'}H_{-\omega'}G_{-\mu'-\epsilon}E_s w}.
\end{aligned}
$$
Since $n > |\lambda+\omega+\mu+\epsilon|+2$, then $s>2$. Thus, $E_s w=\psi(E_s)w=0$, (ii) follows.
\end{proof}

\begin{coro}\label{C 4.4}
Suppose $V$ is a Whittaker module for $\mathfrak{q}_{\epsilon}$, and let $v \in V$.
Regarding $V$ as a $\mathfrak{q}_{\epsilon}^+$-module under the dot action, then $U(\mathfrak{q}_{\epsilon}^+) \cdot v$ is a finite-dimensional submodule of $V$.
\end{coro}

\begin{proof}
This is a direct consequence of Lemma \ref{L 4.3}.
\end{proof}

\begin{theo} \label{T 4.5}
Let $V$ be a Whittaker module for $\mathfrak{q}_{\epsilon}$, and let $S \subseteq V$ be a submodule. Then there exists a nonzero Whittaker vector $w' \in S$, where $w'$ corresponds to the same function $\psi$, for the module $V$ itself.
\end{theo}

\begin{proof}
Regard $V$ as a $\mathfrak{q}_{\epsilon}^+$-module under the dot-action. Let $0 \neq v
\in S$, and let $F$ be the submodule of $S$ generated by $v$ under the dot-action of $\mathfrak{q}_{\epsilon}^+$. By Corollary \ref{C 4.4}, $F$ is a finite-dimensional $\mathfrak{q}_{\epsilon}^+$-module.~Lemma \ref{L 4.3} implies that $E_n \cdot F = 0$ for sufficiently large $n$, so the quotient of $\mathfrak{q}_{\epsilon}^+$ by the kernel of this
action is also finite-dimensional. Note that $E_n$ is locally nilpotent on $V$ (and thus on $F$) under this action by Lemma \ref{L 4.2}.~Thus Engel's Theorem (see \cite{4}) implies that there exists a nonzero $w' \in F \subseteq S$ such that $x \cdot w'=0$ for all $x \in \mathfrak{q}_{\epsilon}^+$. By definition of the dot action, $w'$ is a Whittaker vector.
\end{proof}

\begin{coro}\label{C 4.6}
For any $\xi \in \mathbb{C}$, $\widetilde{L_{\epsilon}(\psi, \xi)}$ is simple.
\end{coro}

\begin{proof}
Let $S$ be a nonzero submodule of $L_{\epsilon}(\psi, \xi)$. Since $C \in \mathfrak{q}_{\epsilon}$ acts by the scalar $\xi$ on $L_{\epsilon}(\psi, \xi)$, it follows from Theorem \ref{T 4.5} that there exists a nonzero Whittaker vector $\overline{w}' \in S$. Theorem \ref{P 3.7} implies that $\overline{w}' = c\overline{w}$ or $cG_{-\epsilon}\overline{w}$ for some $c \in \mathbb{C}$. Since $\overline{w}$ is a cyclic Whittaker module, then $S=L_{\epsilon}(\psi, \xi)$ or $S=\langle G_{-\epsilon}\overline{w}\rangle$. We can get $\langle G_{-\epsilon}\overline{w}\rangle$ lies in every submodule of $L_{\epsilon}(\psi,\xi)$. Therefore, $S/\langle G_{-\epsilon}\overline{w}\rangle$ runs over all submodules of $\widetilde{L_{\epsilon}(\psi, \xi)}$ when $S$ runs over all submodules of $L_{\epsilon}(\psi,\xi)$. However, $S/\langle G_{-\epsilon}\overline{w}\rangle=\widetilde{L_{\epsilon}(\psi, \xi)}~{\rm or}~0$, thus $\widetilde{L_{\epsilon}(\psi, \xi)}$ is simple.
\end{proof}

\begin{theo}\label{C 4.7}
Let $\psi : \mathfrak{q}_{\epsilon}^+ \to \mathbb{C}$ be a Lie superalgebra homomorphism and
$\psi (L_1)$, $\psi (L_2)$, $\psi(H_1) \neq 0$. Let $S$ be a simple Whittaker module of type $\psi$ for $\mathfrak{q}_{\epsilon}$. Then $S \cong \widetilde{L_{\epsilon}(\psi,\xi)}$ for some $\xi \in \mathbb{C}$.
\end{theo}

\begin{proof}
Let $w_s \in S$ be a cyclic Whittaker vector corresponding to $\psi$. It follows from Schur's lemma (Lemma 2.1.3 in \cite{39}) that the center of $U(\mathfrak{q}_{\epsilon})$ acts by a scalar. And this means that there exists $\xi \in \mathbb{C}$ such that $C s = \xi s$ for all $s \in S$. Now, using the universal property of $M_{\epsilon}(\psi)$, there exists a module homomorphism $\varphi : M_{\epsilon}(\psi) \to S$ with $uw \mapsto uw_s$. Since $w_s$ generates $S$, this map is surjective. On the one hand,
$$
\varphi \left((C -\xi 1) M_{\epsilon}(\psi) \right) = (C - \xi 1) \varphi (M_{\epsilon}(\psi)) = (C - \xi 1)S = 0.
$$
On the other hand, $\langle G_{-\epsilon}w\rangle$ is a proper submodule of $M_{\epsilon}(\psi)$, so $\varphi(\langle G_{-\epsilon}w\rangle)=\langle G_{-\epsilon}w_s\rangle$ is a submodule of $S$. However, $w_s \notin \langle G_{-\epsilon}w_s\rangle$, then $\langle G_{-\epsilon}w_s\rangle$ is a proper submodule of $S$. Since $S$ is simple, then $\langle G_{-\epsilon}w_s\rangle=0$.
This implies that $(C - \xi 1) M_{\epsilon}(\psi) \subseteq \ker \varphi$ and $\langle G_{-\epsilon}w\rangle \subseteq \ker \varphi$. Since $\widetilde{L_{\epsilon}(\psi, \xi)}$ is simple and $\ker \varphi \neq M_{\epsilon}(\psi)$, This forces $\ker \varphi = (C - \xi 1) M_{\epsilon}(\psi)\oplus\langle G_{-\epsilon}w\rangle$.Thus, $S\cong \widetilde{L_{\epsilon}(\psi, \xi)}$.
\end{proof}

For a given $\psi : \mathfrak{q}_{\epsilon}^+ \to \mathbb{C}$ and $\xi \in \mathbb{C}$, note that
$$
\begin{aligned}
L_{\epsilon}~&{= U(\mathfrak{q}_{\epsilon})(C - \xi 1) + \sum_{m>0} U(\mathfrak{q}_{\epsilon})(L_m -\psi(L_m)1) +\sum_{m>0} U(\mathfrak{q}_{\epsilon})(H_m - \psi(H_m)1)}\\
&{+\sum_{m\geq0}U(\mathfrak{q}_{\epsilon})G_{m-\epsilon} \subseteq U(\mathfrak{q}_{\epsilon})}
\end{aligned}
$$
is a left ideal of $U(\mathfrak{q}_{\epsilon})$. For $u \in U(\mathfrak{q}_{\epsilon})$, let $ \hat{u}$ denote the coset $u + L_{\epsilon} \in U(\mathfrak{q}_{\epsilon})/L_{\epsilon}$. Then we may consider $U(\mathfrak{q}_{\epsilon})/L_{\epsilon}$ as a Whittaker module of type $\psi$ with cyclic Whittaker vector $\hat{1}$.

\begin{lemm}\label{L 4.8}
Fix $\psi : \mathfrak{q}_{\epsilon}^+ \to \mathbb{C}$ with $\psi(L_1), \psi(L_2), \psi(H_1) \neq0$. Define the left ideal $L_{\epsilon}$ of $U(\mathfrak{q}_{\epsilon})$ as above, and regard $V = U(\mathfrak{q}_{\epsilon}) / L_{\epsilon}$ as a left $U(\mathfrak{q}_{\epsilon})$-module. Then $V$ is simple, and thus $V \cong \widetilde{L_{\epsilon}(\psi, \xi)}$.
\end{lemm}

\begin{proof}
Note that the center of $U(\mathfrak{q}_{\epsilon})$ acts by the scalar $\xi$ on $V$. By the universal property of $M_{\epsilon}(\psi)$, there exists a module homomorphism $\varphi : M_{\epsilon}(\psi) \to V$ with $uw \mapsto u\hat{1}$. This map is surjective because $\overline 1$ generates $V$. But then
$$
\varphi \left((C - \xi 1) M_{\epsilon}(\psi) \right) = (C -\xi 1) \varphi (M_{\epsilon}(\psi) ) = (C - \xi 1)V = 0
$$
and $\varphi(\langle G_{-\epsilon}w\rangle)=\langle G_{-\epsilon}\rangle\hat{1}=0$.
It follows that $(C - \xi1) M_{\epsilon}(\psi) \subseteq \ker \varphi$ and $\langle G_{-\epsilon}w\rangle \subseteq \ker \varphi$. Since $\widetilde{L_{\epsilon}(\psi, \xi)}$ is simple and $\ker \varphi \neq M_{\epsilon}(\psi)$,
this implies $\ker \varphi = (C - \xi 1) M_{\epsilon}(\psi)\oplus\langle G_{-\epsilon}w\rangle$. Thus, $V\cong \widetilde{L_{\epsilon}(\psi, \xi)}$
\end{proof}

Finally, we will give a criterion for the irreducibility of a Whittaker module.

\begin{theo}\label{P 4.9}
Suppose that $V$ is a Whittaker module of type $\psi$ with a cyclic Whittaker vector $w$ such that $C \in \mathfrak{q}_{\epsilon}$ acts by the scalar $\xi \in \mathbb{C}$ and $G_{-\epsilon}w=0$. Then $V$ is simple. Furthermore, ${\rm Ann}_{U(\mathfrak{q}_{\epsilon})} (w) = U(\mathfrak{q}_{\epsilon}) (C-\xi 1) + \sum_{m>0}
U(\mathfrak{q}_{\epsilon})(L_m - \psi(L_m)1)+\sum_{m>0} U(\mathfrak{q}_{\epsilon})(H_m - \psi(H_m)1)+\sum_{m\geq0}U(\mathfrak{q}_{\epsilon})G_{m-\epsilon}$.
\end{theo}

\begin{proof}
Let $K_{\epsilon}$ denote the kernel of the natural surjective map $U(\mathfrak{q}_{\epsilon}) \to
V$ given by $u \mapsto uw$. Then $K_{\epsilon}$ is a proper left ideal containing $L_{\epsilon} =U(\mathfrak{q}_{\epsilon}) (C-\xi 1) + \sum_{m>0}
U(\mathfrak{q}_{\epsilon})(L_m - \psi_m1)+\sum_{m>0} U(\mathfrak{q}_{\epsilon})(H_m - \psi(H_m)1)+\sum_{m\geq0}U(\mathfrak{q}_{\epsilon})G_{m-\epsilon}$. It follows from Lemma \ref{L 4.8} that $L_{\epsilon}$ is maximal. Therefore, $K_{\epsilon}=L_{\epsilon}$ and $V\cong U(\mathfrak{q}_{\epsilon})/L_{\epsilon}$ is simple.
\end{proof}

\begin{rema}
According to Schur's Lemma, Theorem \ref{P 4.9} applies to any simple Whittaker module.
\end{rema}

\vskip30pt \centerline{\bf ACKNOWLEDGMENT}
N. Jing would like to thank the support of
Simons Foundation grant 523868 and NSFC grant No.12171303.
 H. Zhang would
like to thank the support of NSFC grant No.12271332 and Shanghai Natural Science Foundation grant no. 22ZR1424600.


\end{document}